\documentclass[11pt]{article}
\usepackage{graphicx} 

\usepackage[a4paper,margin=2cm]{geometry}
\usepackage{enumerate}

\usepackage{setspace}

\usepackage{pdfpages}

\usepackage{eurosym}

\usepackage{amsmath,amsfonts,amssymb,graphics,amsthm, comment}
\usepackage{hyperref}
\hypersetup{
    colorlinks=true,
    linkcolor=blue,
    citecolor=purple,
    urlcolor=blue,
    pdfborder={0 0 0}
}
\usepackage{enumitem}

\definecolor{mygreen}{rgb}{0,0.2,0.8}

\usepackage{xcolor}
\usepackage{transparent}
\usepackage{bbm}
\usepackage{wrapfig} 

\usepackage[font=sf, labelfont={sf,bf}, margin=2cm]{caption}

\usepackage{cleveref}
  \crefname{theorem}{Theorem}{Theorems}
  \crefname{thm}{Theorem}{Theorems}
  \crefname{thm*}{Theorem*}{Theorems}
    \crefname{problem}{Problem}{Theorems}
  \crefname{lemma}{Lemma}{Lemmas}
  \crefname{lem}{Lemma}{Lemmas}
  \crefname{remark}{Remark}{Remarks}
  \crefname{prop}{Proposition}{Propositions}
\crefname{notation}{Notation}{Notations}
\crefname{claim}{Claim}{Claims}
  \crefname{defn}{Definition}{Definitions}
  \crefname{corollary}{Corollary}{Corollaries}
  \crefname{section}{Section}{Sections}
  \crefname{figure}{Figure}{Figures}
    \crefname{assumption}{Assumption}{Assumptions}

\newtheorem{thm}{Theorem}[section]
\newtheorem{theorem}{Theorem}[section]

\newtheorem{conj}{Conjecture}[section]

\newtheorem{problem}{Problem}[section]
\newtheorem{thm*}{Theorem*}[section]

\newtheorem{lemma}[thm]{Lemma}

\newtheorem{corollary}[thm]{Corollary}
\newtheorem{prop}[thm]{Proposition}
\newtheorem{defn}[thm]{Definition}

\numberwithin{equation}{section}

\theoremstyle{definition}
\newtheorem{remark}[thm]{Remark}

\def\cN{\mathcal{N}}

\def\cL{\mathcal{L}}

\def\cI{\mathcal{I}}

\def\cE{\mathcal{E}}

\def\cC{\mathcal{C}}

\def\cA{\mathcal{A}}

\def\P{\mathbb{P}}
\def\E{\mathbb{E}}
\def\C{\mathbb{C}}
\def\R{\mathbb{R}}

\def\N{\mathbb{N}}

\def  \p- {p\textunderscore}

\def\eps{\varepsilon}

\def\ph{\varphi}

\def\bl{\boldsymbol{\lambda}}
\def\ef{\mathbf{f}}

\usepackage{multicol}

\newcommand{\bdec}[3]{{E}_{#1 \overset{#3}{\to} #2}}

\newcommand{\dd}{\mathrm{d}}

\def\bl{\boldsymbol{\lambda}}
\def\ef{\mathbf{f}}
\def\bmu{\boldsymbol{\mu}}

\title{On the spectral geometry of Liouville quantum gravity}
\author{Nathana\"el Berestycki\\
\small{Universität Wien, Fakultät für Mathematik. }\\
\small{Oskar Morgenstern Platz, Wien 1090 -- Austria}\\
\small{ \texttt{nathanael.berestycki@univie.ac.at}}}
\date{December 2025}

\begin{document}

\maketitle

\begin{abstract}
We give a concise presentation of the construction of the Liouville quantum gravity (LQG) eigenvalues and eigenfunctions, i.e., the spectrum associated to the infinitesimal generator of Liouville Brownian motion, the canonical diffusion in the geometry of LQG. We describe the recently obtained Weyl law in this context \cite{BW} (giving a short summary of its proof) and report on some work in progress concerning the associated heat trace \cite{BK}. Finally, we summarise and propose some new key open problems in this direction. 
\end{abstract}
\tableofcontents

\section{Overview and motivation}

Liouville quantum gravity (LQG thereafter), also known as Liouville conformal field theory (CFT), aims to give a mathematically rigorous framework to the theory proposed by Polyakov \cite{Polyakov} for 2D quantum gravity, which is based on the formal action
\begin{equation}\label{E:Polyakov}
  S(\ph) = \frac1{4\pi} \int_{\Sigma} \Big[|\nabla^g \ph(z)|^2 + R_g Q \ph(z)  + 4 \pi \lambda e^{\gamma \ph(z)}  \Big] v_g( \dd z),
\end{equation}
where $(\Sigma,g)$ is a Riemann surface (which could be a domain or, in the case that concerns us below, the Riemann sphere), $\ph:\Sigma \to \R$ is a field (e.g. a distribution) on this surface, $\nabla^g \ph$ denote the gradient of $\ph$ with respect to the metric tensor $g$, $v_g$ is the associated volume form, $R_g$ is the scalar curvature, $\gamma \ge 0$ is the fundamental parameter of the theory (known as the {coupling constant}), $Q = (2/\gamma) + (\gamma/2)$, and $\lambda>0$ is a (less important) parameter known as cosmological constant.

The expression \eqref{E:Polyakov} is fundamentally of a geometrical nature and is closely related to the notion of \textbf{conformal factor}. Roughly speaking, one should think of the field $\ph: \Sigma \to \R$ appearing in \eqref{E:Polyakov} as inducing a new metric $g_\ph$ on $\Sigma$, obtained by locally rescaling the metric $g$ at each point of $\Sigma$ in such a way that (at least if $\ph$ were actually a smooth function on $\Sigma$), the associated volume form $v_\ph$ is $e^{\gamma \ph} v_g$.  
In reality, we will see below that a sample of the field $\ph$ associated to the action \eqref{E:Polyakov} will turn out to not even be defined pointwise and must instead be understood as a (random) distribution, i.e., a (random) generalised function. Such a field $\ph$ is therefore \emph{a priori} much too rough for the definition of $g_\ph$ to make litteral sense.

Assigning a rigorous meaning to the above expression, and more precisely to the measure on fields represented by the formal expression $\mathbf{P}(\dd\ph) = \exp ( - S(\ph)) \dd \ph$ (where $\dd \ph$ is a formal uniform measure on fields), is a nontrivial task, 
which was successfully resolved in the groundbreaking work by David, Kupiainen, Rhodes and Vargas \cite{DKRV} with some important preceding contributions by Duplantier and Sheffield \cite{DS}. While it would be futile to try and give here a complete overview of the state of the art in this field, we mention here briefly a few important directions of research which have been spectacularly successful in the last few years (for further details we refer the reader for instance to the book \cite{BP} which attempts to survey a number of these developments).  

\medskip (1) First, starting with \cite{DKRV} itself, a remarkable sequence of works 
including for instance \cite{DKRV, DOZZ,DS, GKRV} explored a deep connection to conformal field theory. In particular these authors were able to 
show that the associated correlation functions can be studied to great effect using the tools of conformal field theory (which are analytic and algebraic in essence), resulting notably in surprising, exact yet highly nontrivial, formulae; a particularly notable breakthrough in this direction is the proof of the celebrated DOZZ formula in \cite{DOZZ}. 

(2) In parallel, a powerful connection to Schramm--Loewner Evolutions (SLE) was revealed and exploited in works such as \cite{DS, LQGmating, Mliouville1, HoldenSun}. These authors were partly motivated by the conjectured connection to random planar maps (\cite{LeGall, Miermont}) and more generally planar maps decorated by a critical model of statistical mechanics such as the self-dual Fortuin--Kasteleyn percolation model. In this direction we mention in particular the remarkable work 
 \cite{HoldenSun} in which certain embeddings of random planar maps were shown to converge in an appropriate sense to LQG with coupling constant $\gamma = \sqrt{8/3}$ in the scaling limit. 
 
 (3) A series of works (see e.g. \cite{AHS, FZZ, AngRemySun, Backbone}) combined these two points of views in order to derive beautiful, \emph{exact} formulae related to SLE curves and their loop ensembles, eventually proving a number of prominent predictions coming from statistical mechanics and establishing the value of some critical exponents which were previously inaccessible through purely SLE-based methods. 
 
(4) Finally, a random metric associated with fields such as those in \eqref{E:Polyakov} was constructed in another remarkable series of works \cite{GM, DDDF}, generalising the earlier construction in \cite{Mliouville1} in the case $\gamma=\sqrt{8/3}$, and which is isometric to the so-called Brownian map (i.e., the scaling limit in a metric sense of uniform random planar maps, as proved in \cite{LeGall, Miermont}).

\medskip Despite the above remarkable successes, many basic questions concerning the geometry of LQG remain to date quite mysterious. 
Do classical notions of geometry have an LQG analogue? What about classical theorems of geometry? We choose to focus here on the \textbf{spectral geometry of LQG}, i.e., the study of associated eigenfunctions and eigenvalues (which we will define more precisely below). We ask questions such as: \emph{can one hear the shape of LQG?} (i.e., do the eigenvalues a.s. determine the underlying field $\ph$?), or \emph{are the eigenfunctions localised or delocalised?} 
This is not only mathematically natural, but, as we will see, the conjectures we make below suggest a rather fascinating connection with the field of \textbf{quantum chaos}.

\section{Background}

\subsection{Brief overview of Liouville quantum gravity}

The starting point of both \cite{DS} and \cite{DKRV} is the {Gaussian free field} on $\Sigma$. (When $\Sigma$ has a boundary, as in the case of a domain, then it is natural to impose either Dirichlet or Neumann boundary conditions for the Green function; otherwise, as in the case of the sphere, it is natural to require the field to have mean zero).  In order to assign a meaning to the exponential term in \eqref{E:Polyakov}, a prominent role in the theory is played by the random measure $\mu = \mu_h$ associated to a realisation $h$ of the Gaussian free field and formally defined by the expression
$
\mu(\dd z) =  \exp ( \gamma h(z) )) v_g(\dd z),
$
where $\gamma$ is the {coupling constant} from \eqref{E:Polyakov}. In reality the preceding expression is only formal, since $h$, as a random distribution, is too rough for its exponential to be directly well defined. Instead one must rely on the theory of {Gaussian multiplicative chaos} (\cite{K85, BerestyckiGMC}), in which we consider $\eps$-regularisations $h_\eps(z) = h \star \theta_{z,\eps}$ of $h$, where $\theta_{z, \eps}$ is the image of the regularisation kernel $\theta$ under translation and scaling which maps $B(0,1)$ to $B(z, \eps)$, and the regularisation kernel $\theta$ is a fairly arbitrary signed measure integrating the Green function (see \cite[Chapter 3]{BP}), or slightly more generally an element of the Sobolev space $H^{-1}(\Sigma)$. The idea is to show:

\begin{thm}\label{T:GMC}
\begin{equation}\label{eq:GMC1}
\mu_h( \mathrm{d} z) = \lim_{\eps\to 0} \eps^{\gamma^2/2} \exp \left( \gamma h_\eps( \dd z)  \right) v_g (\dd z)
\end{equation}
exists in probability.  Furthermore the resulting measure is nonzero precisely when $\gamma \in [0,2)$, and is independent of the choice of regularisation kernel $\theta$ subject to very mild conditions. 
\end{thm}

(The convergence in the above result is for the weak convergence of measures.) This random measure is sometimes known as the Liouville volume measure. 
See \cite{RhodesVargas_survey} for a survey and \cite[Chapter 3]{BP} for a thorough introduction.


\medskip With this in hand, the authors of \cite{DKRV} simply define the measure $\mathbf{P}( \dd \ph)$ associated with \eqref{E:Polyakov} to be the ``law'' of $\ph = h + c$, where $h$ is a GFF on $\Sigma$ and $c$ (the ``zero mode'') is sampled from Lebesgue measure on $\R$ (we use quotation marks since this is an infinite measure), reweighted by $\exp ( - \tfrac1{4\pi} \int_\Sigma Q R_g \ph v_g - \mu \mu_\ph(\Sigma))$. In other words, they set
\begin{equation}\label{E:Polyakov2}
\int F(\ph) \mathbf{P} ( \dd \ph) = \int_{\R} \E^{\mathrm{GFF}}  \left[ F( h+c)  \exp ( - \frac{1}{4\pi} \int_\Sigma Q R_g (h+c) \dd v_g  - \lambda \mu_{h + c} (\Sigma) )  \right] \dd c. 
\end{equation}
The connection with \eqref{E:Polyakov} is that we view $\exp ( - \frac{1}{4\pi} \int_\Sigma |\nabla^g \ph |^2 \dd v_g ) \dd \ph$ as the law of $\ph:=h+c$, and we view the remaining terms in the formal expression $\exp ( -S(\ph)) \dd \ph$ as a Radon--Nikodym derivative with respect to this law. This leads to the expression \eqref{E:Polyakov2}.

It is common practice to denote the integral on the left hand side as $\langle F \rangle = \langle F\rangle_g$.  Thus, 
noting that $(h, R_g) = 0$ in the case where $g$ has constant curvature, and that $(c, R_g ) = 8\pi c$ by the Gauss--Bonnet theorem for any constant $c$, we have 
\begin{equation}\label{eq:Polyakovmeasure}
\langle F \rangle_g: = \int_\R \E[ F( \ph )  \exp ( - 2 Qc  - \lambda \mu_{\ph}   ( \Sigma) ) ]  \dd c ; \quad \ph = h + c.
\end{equation}
In particular, taking $g$ to be the standard round metric on the Riemann sphere, the {correlation functions} at points (insertions) $\mathbf{z} = (z_1, \ldots, z_n) $ and weights ${\alpha} = (\alpha_1, \ldots, \alpha_n )\ge 0$ are by definition given by
\begin{equation}\label{E:correl}
Z_{\alpha; \mathbf{z}} = \langle e^{\alpha_1 \ph(z_1) + \ldots  + \alpha_n \ph(z_n)} \rangle_g := \int_\R \E[ e^{\alpha_1 \ph(z_1) + \ldots + \alpha_n \ph (z_n) } \exp( - 2 Qc  - \lambda \mu_{\ph } ( \Sigma))  ] \dd c.
\end{equation}
In the above expression, a suitable regularisation must again be taken as $\ph(z_1), \ldots, \ph(z_k)$ are not pointwise defined. The authors of \cite{DKRV} proved the fundamental fact that these correlations are well defined if and only if the weights $\alpha_i$ satisfy the inequalities:
\begin{equation}\label{E:Seiberg}
\sum_{i=1}^n \alpha_i > 2Q, \quad \alpha_i < Q \quad (i = 1, \ldots, n),
\end{equation}
a set of conditions known as the {Seiberg bounds}. Roughly speaking, the role of the first condition is to ensure that the surface does not degenerate to a trivial surface (with zero area), while the second half of these conditions is a compactness condition ensuring that the volume near $z_i$ does not blow up. See \cite{DKRV} for details or \cite[Chapter 5]{BP} for a discussion. Note that \eqref{E:Seiberg} implies that the number $n$ of insertions must be at least 3. In particular, $\langle 1 \rangle = \infty$, i.e., $\mathbf{P}$ is not a finite measure. We obtain however a well defined probability measure by specifying a number of pairwise disjoint insertion points $z_1, \ldots, z_n \in \Sigma$ and associated weights $\alpha_1, \ldots, \alpha_n$ satisfying the Seiberg bounds \eqref{E:Seiberg}, and considering the law whose partition function is $Z_{\alpha; \mathbf{z}}^\gamma$: that is, for nonnegative measurable functionals $F$ (on, say, $H^{-1} ( \Sigma)$), we set
\begin{equation}\label{E:Polyakovcorrel}
\E_{\alpha; \mathbf{z}} ( F( \ph) ) = \frac{ \langle e^{\alpha_1 \ph (z_1) + \ldots + \alpha_n \ph(z_n) } F(\ph) \rangle}{Z_{\alpha; \mathbf{z}}}.
\end{equation}

Again, see \cite[Chapter 3]{BP} for details concerning the definition of \eqref{E:correl} and \eqref{E:Polyakovcorrel}. 
 We may think informally of the field $\ph$ as endowing $\Sigma$ with a ``random Riemannian structure''. Its associated ``volume measure'' is the finite random measure $\mu_\ph$ defined in \eqref{eq:GMC1} (with $\ph$ instead of $h$), which is then a.s. well-defined.

\medskip While the definition of the law $\P_{\alpha; \mathbf{z}}$ appears at first sight involved, it is necessary to work with this in order to see the rich integrable structure behind LQG. On the other hand, many qualitative features of LQG can be understood by considering the \textbf{toy model} where $\Sigma$ is taken to be a proper simply connected domain of the plane, and $\ph$ is simply a Gaussian free field with Dirichlet (or zero) boundary conditions (see, e.g., \cite[Chapter 1]{BP}); this was for instance the case considered in the landmark paper \cite{DS}. This is also what we do in most of what follows.


\subsection{Liouville Brownian motion}

Informally speaking, the eigenfunctions and eigenvalues which we aim to study in this note are those of an appropriate \textbf{Laplace--Beltrami operator} associated to LQG. In reality, this operator is somewhat cumbersome to describe in functional analytic terms, so we start by describing instead the Markov process associated to this generator. This is known under the name of \textbf{Liouville Brownian motion}. It can be viewed as the canonical diffusion in the random Riemannian geometry defined by $\ph$ on $\Sigma$ (whether $\ph$ has the law $\P_{\alpha; \mathbf{z}}$ above or is simply a Dirichlet GFF on the proper simply connected domain $\Sigma$ of the complex plane).
Its definition and construction, which dates back to the parallel works  \cite{diffrag, GRV}, starts with an ordinary Brownian motion $(X_t)_{t \ge 0}$ on the surface $\Sigma$ with respect to the metric tensor $g$ (or, in the toy model, an ordinary Brownian motion in the plane, killed upon leaving the domain $\Sigma$), and starting from a given point $x \in \Sigma$. 

\begin{defn}\label{D:LBM}
The Liouville Brownian motion $(Z_t)_{t\ge 0}$ is the process obtained as a time-change of $X$, namely $Z_t = X_{F^{-1}(t)}$, where $F^{-1} (t)$ denotes the inverse of the additive functional of $X$ (the so-called quantum clock) given by 
\begin{equation}\label{eq:clock}
F(t) = \int_0^t e^{\gamma \ph(X_s)} \dd s := \lim_{\eps \to 0} \int_0^t \eps^{\gamma^2/2} e^{\gamma \ph_\eps(X_s)  } \dd s.
\end{equation}
\end{defn}

The existence of this limit can be seen as following from the theory of Gaussian multiplicative chaos, and is nontrivial (i.e., nonconstant) for $\gamma\in [0,2)$, corresponding to the range for which LQG itself is nontrivial (recall Theorem \ref{T:GMC}). Again, the limit is independent of the choice of regularisation kernel $\theta$ subject to same mild conditions as in Theorem \ref{T:GMC}; see again \cite{BerestyckiGMC} and \cite[Chapter 3]{BP}.  

\medskip Thus, as a time-change of ordinary Brownian motion, Liouville Brownian motion is completely isotropic: it has no preferred direction. The effect of the random field $\ph$ is only through the time-parametrisation of $Z$: this can be viewed as a consequence of conformal invariance of Brownian motion in two dimensions. The time-change is such that $Z$ moves ``slowly'' (compared to $X$) whenever $Z$ is at a point where $\ph$ is large. This is related to the fact that near such points the area $\mu_\ph$ is also large: thus it is reasonable that it should take time for the diffusion to leave what is actually a ``big region'' for the intrinsic geometry defined by $\ph$.  

\medskip We will assume for concreteness in the rest of the article that we work with the toy model where $\Sigma$ is a bounded domain $\Sigma$ of the complex plane and $\ph$ is the Gaussian free field on $\Sigma$ with Dirichlet boundary conditions, but we do not expect any significant difference if $\ph$ has the more complex law specified by \eqref{E:Polyakovcorrel}.

\medskip It can be shown that the time-change $F^{-1}$ is continuous and strictly increasing whenever $0\le \gamma <2$. Furthermore, it is shown in \cite{GRV} that it is possible to extend Definition \ref{D:LBM} so as to give an almost sure meaning to $Z$ started from a given point $x \in \Sigma$, but from \emph{all} points $x \in \Sigma$ simultaneously. In other words, let us call $\mathbf{P}_x$ the law of $Z$ starting from $x \in \Sigma$. Then, on an event of probability one for the field $\ph$, the collection of laws $(\mathbf{P}_x)_{x\in \Sigma}$ makes sense for all points $x\in \Sigma$ simultaneously. Furthermore, \cite{GRV} shows that this collection of law defines (again, on an event of probability one for the field), a Feller process on $\Sigma$, and in fact a strongly Feller process: that is, its semigroup $\mathbf{P}_t$ defined almost surely for all $x \in \Sigma$ by $\mathbf{P}_t f(x) = \mathbf{E}_x [f(Z_t)]$ maps bounded Borel measurable functions to continuous bounded functions on $\Sigma$ (see for instance Theorem 5.1 in \cite{AndresKajino}). It is also strongly continuous on $L^2 (\mu_\ph)$, a.s. (see Theorem 2.18 in \cite{GRV}): that is, $\mathbf{P}_t f$ is well defined (and belongs to $L^2(\mu_\ph)$) as soon as $f \in L^2 (\mu_\ph)$; furthermore, $\mathbf{P}_t f \to f$ in $L^2 (\mu_\ph)$ as $t\to 0$. This semigroup is furthermore \textbf{symmetric}: that is, for all Borel nonnegative functions $f,g: \Sigma \to [0, \infty)$, we have
\begin{equation}\label{eq:Ptsym}
\int f \cdot( P_t g) \  \dd \mu_\ph = \int (P_t f) \cdot g\   \dd \mu_\ph .
\end{equation}
 See, e.g., Theorem 2.2 in \cite{AndresKajino}.

\medskip Although this will not be used further below, it was observed in \cite{GRV_hk} and \cite{AndresKajino} that Liouville Brownian motion can be defined in terms of Dirichlet forms \cite{Fukushima}, where the relevant Dirichlet form $\cE(f,g)$ is nothing but $\int_\Sigma \nabla f \cdot \nabla g $ on the Sobolev space $H_0^1 (\Sigma)$, viewed as a (strongly local regular, symmetric) Dirichlet form on $L^2 ( \mu_\ph)$. That is, $\mu_\ph$ is a reversible measure for Liouville Brownian motion. In other words, and using further the vocabulary and notions from Dirichlet forms theory \cite{Fukushima}, Liouville Brownian motion is an example of Markov process defined by the \textbf{Revuz correspondence} for which $\mu_\ph$ is the Revuz measure. 

\medskip Liouville Brownian motion is conjectured to be the \textbf{scaling limit} of simple random walk on many families of random planar maps under suitable (discrete conformal) embeddings. Such a statement was proved in joint work with Gwynne \cite{BerestyckiGwynne} on the Tutte embeddings of random planar maps known as CRT-mated planar maps. The latter are canonical discretisations of Liouville quantum gravity defined using the so-called \emph{mating of trees} framework developed by Duplantier, Miller and Sheffield \cite{LQGmating}. This clarifies the sense in which Liouville Brownian motion is the canonical diffusion in the geometry of LQG. It is worth noting that the result of \cite{BerestyckiGwynne} is based on an axiomatic characterisation of Liouville Brownian motion.

\subsection{Green function, spectrum}
\label{SS:Green_spectrum}

In this section we explain how the Green function, heat kernel and spectrum of Liouville Brownian motion can be defined rigorously; this follows roughly the approach in (\cite{AndresKajino}, \cite{MRVZ}, \cite{GRV_hk}) and can be considered background material for this article. It can thus safely be skipped by a reader eager to get to the newer material.

\medskip In the following we will require the following basic and easy estimate for the uniform modulus of continuity of $\mu_{\ph}$:

\begin{lemma}\label{L:modcont}
Let $\gamma<2$  and $\mu_\ph$ be as above. Then there exists $C = C(\omega) < \infty$ and $q >0$ nonrandom such that 
\begin{equation}\label{eq:largeballs}
\sup_{x \in \Sigma} \mu_\ph ( B (x,r)) \le C r^{q}.
\end{equation}
\end{lemma}
See for instance, Exercise 3.4 in \cite{BP} (alternatively, Lemma 3.1 in \cite{AndresKajino} or Lemma 3.1 in \cite{MRVZ}, for a proof).

\medskip As mentioned above, the infinitesimal generator of Liouville Brownian motion, although well defined abstractly for instance from Dirichlet form theory, is not especially convenient to manipulate -- the difficulty comes from the domain on which the generator is defined, which is not obvious to describe explicitly. However, its associated \textbf{Green function}, formally the inverse of the generator, is in fact well defined and easy to describe; this will make it simple to apply the spectral theorem to it in order to define the spectrum of LQG. To some extent this mirrors the situation in the standard Riemannian/Euclidean setup: indeed the negative Laplacian $-\Delta$ is not directly a compact operator (for instance, \emph{a posteriori}, its eigenvalues will tend to infinity, as the standard Weyl law indicates); however its inverse, namely the integral operator whose kernel is the standard Green function $g_\Sigma (x,y)$, \emph{is} compact, and indeed its eigenvalues (which are the inverses of those of $-\Delta$) tend to 0 and are thus bounded.

\medskip Thus consider Liouville Brownian motion, associated to a Gaussian free field $\ph$ on a domain $\Sigma$ with Dirichlet boundary conditions. From the definition, it is not hard to guess one has the following \textbf{occupation formula}:
\begin{lemma}\label{L:occ}
For a Borel nonnegative function $f:\Sigma \to \R$, if $\mathbf{T}_\Sigma = \inf\{ t\ge 0: Z_t \notin \Sigma\}$,
\begin{equation}\label{eq:occupation}
\mathbf{E}_x[ \int_0^{\mathbf{T}_\Sigma} f(Z_t) \dd t ] = \int_0^\infty g_\Sigma(x,y) f(y) \mu_\ph(\dd y).
\end{equation}
\end{lemma}

Here $g_\Sigma(x,y)$ is the Green function of ordinary Brownian motion killed upon leaving $\Sigma$ (see, e.g., \cite[Chapter 1]{BP}), which satisfies
\begin{equation}\label{eq:Green}
g_\Sigma(x,y) = - \frac1{2\pi} \log | x - y| + O(1).
\end{equation}
The occupation measure formula can be seen as an example of the Revuz correspondence between the process $Z$ and its Revuz measure $\mu_\ph$. It expresses the fact that $Z$ is a time-change of ordinary Brownian motion, where the time change is locally given by $\mu_\ph(\dd y)$ when the process is at $y \in \Sigma$. For instance, if $f = 1_{B(y,r)}$ is the indicator of a small ball near $y$, then the left hand side counts the expected amount of time (given the GFF, averaged just over the trajectories of the process) spent by Liouville Brownian motion in $B(y, r)$. Since only the time parametrisations of $Z$ and $X$ differ from one another, it is reasonable to expect that this expected time is given by $g_\Sigma(x,y) \mu_{\ph} (B(y, r))$ approximately. 

\begin{proof}[Proof of Lemma \ref{L:occ}]
To prove \eqref{eq:occupation}, consider $\eps$-regularisations of $\mu$ and $Z$ and let $\eps\to 0$: if $f\ge 0$ is continuous and bounded on $\Sigma$, $f(Z_t)$ is a.s. continuous as a functional on path space. Therefore, by the portmanteau theorem, a.s., $f(Z_t^\eps) \to Z_t$ as $\eps\to 0$, where $Z^\eps_t = X_{F_\eps^{-1}(t)}$ and $F_\eps(t) = \int_0^t \eps^{\gamma^2/2} e^{\gamma \ph_\eps(X_s)} \dd s .$ For the same reason (and since $\mathbf{T}_\Sigma$ is an a.s. continuous functional on path space), if we let $\mathbf{T}^\eps_\Sigma $ denote the first time that $Z^\eps$ leaves $\Sigma$, we first observe that $\mathbf{T}^\eps_\Sigma \to \mathbf{T}_\Sigma<\infty$ a.s. In fact, since $\sup_{\eps>0} \E((\mathbf{T}^\eps_\Sigma)^p) \le 1$ for some $p>1$ (see \cite[Proposition 3.37]{BP}), we deduce that  $\mathbf{E}_x ( \mathbf{T}^\eps_\Sigma) \to \mathbf{E}_x( \mathbf{T}_\Sigma)$, a.s. As a consequence, we have the following a.s. equalities:
\begin{align*}
\mathbf{E}_x[ \int_0^{\mathbf{T}_\Sigma} f(Z_t) \dd t] & = \lim_{\eps \to 0} \mathbf{E}_x[ \int_0^{\mathbf{T}^\eps_\Sigma} f(Z^\eps_t) \dd t]\\
& = \lim_{\eps\to 0}
\mathbf{E}_x [ \int_0^{T_\Sigma} f(X_s) \dd F_\eps(s) ] \quad  \text{ (where $T_\Sigma = \inf\{ s>0: X_t \notin \Sigma\}$)}\\
& = \lim_{\eps\to 0}
\mathbf{E}_x[  \int_0^{T_\Sigma} f(X_s) \eps^{\gamma^2/2} e^{\gamma \ph_\eps(X_s)} \dd s] \\
& = \lim_{\eps\to 0}
 \int_\Sigma g_\Sigma (x,y) f(y)   \eps^{\gamma^2/2} e^{\gamma \ph_\eps(y) }\dd y  \\
 & = \int_\Sigma g_\Sigma(x,y) f(y) \mu_\ph(\dd y)
\end{align*}
as desired. The last step of the proof requires some justification since $g_\Sigma$ is e.g. not continuous near $y = x$. On the other hand, the blow-up of $g_\Sigma$ near the diagonal is only logarithmic, see \eqref{eq:Green}, so that in combination with Lemma \ref{L:modcont} the last line above is indeed justified.
This completes the proof of Lemma \ref{L:occ} (see also \cite[Proposition B.1]{AndresKajino} for a somewhat more complicated proof).
\end{proof}

The occupation measure justifies the following definition of the Liouville Green function (or more appropriately Green operator since it is in fact not pointwise defined).
\begin{defn}
For a Borel nonnegative function $f:\Sigma\to \R$ we define 
$$
\mathbf{G} f(x) = \mathbf{E}_x [ \int_0^{\mathbf{T}_\Sigma} f(Z_t)\dd t ] = \int_\Sigma f(y) g_\Sigma(x,y) \mu_\ph (\dd y).
$$
In other words it is the operator associated to the integral kernel $g_\Sigma(x,y)$, but with respect to $\mu_\ph$ rather than the Lebesgue measure $\dd y$. 
\end{defn}

By the Cauchy--Schwarz inequality, and using the fact that almost surely $\sup_{x\in \Sigma} \int_\Sigma g_\Sigma^2(x,y) \mu_\ph(dy) <\infty$ (which can be seen from Lemma \ref{L:modcont}), we see that $\mathbf{G}$ can also a.s. be seen 
 as a linear operator on $L^2 (\mu_\ph)$. Indeed, on the above event, for \emph{all} $f \in L^2 (\ph)$ and for \emph{all}  $x\in \Sigma$,
$$
\left(\int_\Sigma g_\Sigma(x,y) |f(y) | \mu_\ph(\dd y)\right)^2 \le \int_\Sigma g_\Sigma (x,y)^2 \mu(\dd y) \int_\Sigma f(y)^2 \mu_\ph ( \dd y) < \infty.
$$ 
Thus, on this event of probability one, $\mathbf{G}f(x)$ is also defined for all $f\in L^2 (\mu_\ph)$ and all $x \in \Sigma$ simultaneously. In fact, we will see below that $\mathbf{G}$ maps a function $f \in L^2 ( \mu_\ph)$ to a function $\mathbf{G}f$ also in $L^2 ( \mu_\ph)$. 

\medskip Let us review a few fundamental aspects of spectral theory which will be relevant for what follows. Although this is very standard material we have decided to include it here as it is central to the question of how one can associate a spectrum to a diffusion as rough as Liouville Brownian motion and thus to our later work.

\medskip A linear operator $T:X \to Y$ (where $X,Y$ are two normed vector spaces) is called bounded if it has finite operator norm, i.e., if there exists $M>0$ such that $\|Tx\|_Y \le M \|x \|_X$. Such an operator is thus continuous. $T$ is called compact if it maps bounded sets of $X$ to relatively compact sets in $Y$ (such an operator is necessarily bounded). Now suppose $H$ is a Hilbert space and $T: H \to H$ is linear. It is called self-adjoint if $\langle Tx, y\rangle_H = \langle x, Ty\rangle_H$ for all $x,y \in H$. A commonly stated version of the spectral theorem is the following:

\begin{theorem}
Suppose $H$ is an infinite-dimensional, separable Hilbert space and $T: H \to H$ is self-adjoint and compact. Then there exists an orthonormal basis $(e_n)_{n\ge 0}$ consisting of eigenvectors of $T$, whose associated eigenvalue $\mu_n$ is real and $\mu_n \to 0$. 
\end{theorem}

In the above theorem, $\mu$ is an eigenvalue if there is a nonzero $f$ such that $Tf = \mu f$. Note that nothing prevents $\mu_n $ from being equal to zero. In fact, zero eigenvalues play a slightly special role: for instance, by the Riesz--Schauder theorem, all the nonzero eigenvalues of a compact operator are necessarily finite-dimensional: that is, if $T: H \to H$ is compact and $\mu$ is eigenvalue, then $\dim \ker ( T - \mu I) < \infty$.  

 We will aim to show that we can apply the spectral theorem to $\mathbf{G}$ acting on the Hilbert space $H = L^2(\mu_\ph)$ (we will show separately that zero cannot be an eigenvalue). 
Our first task is to check compactness. For this, we will rely on the notion of a Hilbert--Schmidt operator. Recall that the operator $T: H\to H$ is called \textbf{Hilbert--Schmidt} if
$$
\|T \|_{\mathrm{HS}}^2 = \sum_{i=1}^\infty \|Te_i\|_H^2 < \infty 
$$
for some orthonormal basis $(e_i)_{i\ge 1}$ of $H$; the norm $\|T\|_{\mathrm{HS}}$ is then independent of the choice of basis of $H$. Note that if $T$ is Hilbert--Schmidt then for $x = \sum_i x_i e_i \in H$ one has 
\begin{align}
\|T x\|_H &= \|\sum_{i=1}^\infty x_i T(e_i)\|_H   \le \sum_{i=1}^\infty |x_i|\cdot \|T(e_i)\|_H \nonumber \\
&\le \left( \sum_{i=1}^\infty x_i^2 \cdot \sum_{i=1}^\infty \|T(e_i)\|_H^2 \right)^{1/2}\nonumber \\
&=  \|x\|_H \|T\|_{\mathrm{HS}} \label{HSop}
\end{align}
so that the operator norm of $T$ is bounded by the Hilbert--Schmidt norm of $T$. In particular, such an operator is then necessarily compact as the limit (with respect to the Hilbert Schmidt norm, hence also in the operator norm) of bounded finite rank operators.  

An important example of Hilbert--Schmidt operators is provided by Hilbert-Schmidt integral operators: suppose $(X, \cA, \mu)$ is a measure space and $g: X \times X \to \R$ is a measurable function. Then $g$ induces an operator $T$ (let us not specify precisely the space on which $T$ acts at the moment) via the formula $Tf(x) = \int_X g(x,y) f(y) \mu(\dd y)$, whenever the integral in the right hand side is defined. 

\begin{lemma}\label{L:HS} Suppose $g \in L^2 ( \mu \otimes \mu)$. Then $T$ maps $H = L^2(\mu)$ to itself and furthermore $T$ is a Hilbert--Schmidt operator on $H$, with $\|T\|_{\mathrm{HS}} = \|g\|_{L^2(\mu \otimes \mu)}$. 
 In particular $T$ is compact. 
\end{lemma}

\begin{proof}

Fix $(e_i)_{i\ge 1}$ an orthonormal basis of $H = L^2(\mu)$. We have by the positive case of Fubini's theorem (note that, a priori, any of the terms could be infinite),
\begin{align*}
\|T\|_{\mathrm{HS}}^2 & = \sum_{i=1}^\infty \| T e_i\|_H^2 \\
& = \sum_{i=1}^\infty \int_X \mu(\dd x) |T e_i(x)|^2 \\
& =  \int_X \mu(\dd x) \sum_{i=1}^\infty | \langle g(x, \cdot), e_i\rangle_H |^2  \\
& = \int_X \mu(\dd x) \| g(x, \cdot)\|_H^2
\end{align*}
by Parseval's indentity. Now, $\| g(x, \cdot)\|_H^2 = \int_X g(x,y)^2 \mu( \dd y)$. Thus
$$
\|T\|_{\mathrm{HS}}^2 = \iint_{X\times X} \mu(\dd x) g(x,y)^2 \mu ( \dd y) = \| g\|_{L^2(\mu\otimes \mu)}^2,
$$
as desired, and in particular $\|T\|_{\mathrm{HS}}^2< \infty$. In particular, $\|Te_i\|_H^2 <\infty$ so $Te_i \in L^2(\mu)$. Not only this, but if $f = \sum_i a_i e_i \in H$, then the argument in \eqref{HSop} shows that $\| Tf \|_H < \infty$ so $Tf \in H $ and so $T$ really does map $H$ to $H$. 
\end{proof}

\begin{corollary} Let $\gamma <2$. The Liouville Green operator $\mathbf{G}$ is a Hilbert--Schmidt (hence compact) and self-adjoint operator from $L^2 ( \mu_\ph)$ to itself. 
\end{corollary}

\begin{proof} By Lemma \ref{L:HS} it suffices to check that $g_\Sigma \in L^2 (\mu_\ph \otimes \mu_\ph)$. However this follows at once from  Lemma \ref{L:modcont} and the logarithmic divergence of $g_\Sigma$ on the diagonal \eqref{eq:Green}.  It remains to check that $\mathbf{G}$ is self-adjoint, but this is easy to check. Indeed, since $g_\Sigma(x,y) = g_\Sigma(y,x)$ is reversible for the ordinary Brownian motion, if $f,g \in L^2(\mu_\ph)$, 
\begin{align*}
\langle \mathbf{G}f , g\rangle_{L^2 (\mu_\ph)} &= \int_\Sigma \left(\int_\Sigma g_\Sigma(x,y) f(y) \mu_\ph(\dd y) \right) g(x) \mu_\ph(\dd x)\\
& = \iint g_\Sigma(x,y) f(x) g(y) \mu_\ph(\dd x) \mu_\ph(\dd y) = \langle f, \mathbf{G} g\rangle_{L^2(\mu_\ph)}.
\end{align*}
This concludes the proof of the corollary.
\end{proof}

\medskip Altogether, the spectral theorem therefore shows that there exists an orthonormal basis $(\boldsymbol{\mu}_n)_{n\ge 1}$ of eigenvectors in $L^2(\mu_\ph)$ such that each eigenvalue $\boldsymbol{\mu}_n$ is real and $\boldsymbol{\mu}_n \to 0$. These eigenvalues are associated to eigenfunctions $(\ef_n)_{n\ge 1}$ which form an orthonormal basis of $L^2 ( \mu_\ph)$.  We will now show that the eigenvalues are furthermore nonnegative and in fact nonzero (a possibility which is obviously not precluded by the spectral theorem).

\begin{lemma} Almost surely, $\mathbf{G}$ is a nonnegative operator on $L^2 (\mu_\ph)$. That is, almost surely, for all $f \in L^2 ( \mu_\ph)$ we have $\langle f, \mathbf{G}f \rangle_{L^2(\mu_\ph)} \ge 0$. In particular all the eigenvalues $\bmu_n $ of $\mathbf{G}$ are nonnegative. Furthermore, $\ker (\mathbf{G}) = 0$ a.s., thus $\bmu_n >0$ for all $n\ge 0$, a.s.
\end{lemma}

\begin{proof}
Let $f \in L^2 ( \mu_\ph)$. 
Note that by Lemma \ref{L:occ} and since $\mathbf{G}f \in L^2 (\mu_\ph)$ we can write 
$$
\mathbf{G}f (x) = \mathbf{E}_x ( \int_0^{\mathbf{T}_\Sigma}  f( \mathbf{Z}_t) \dd t ) = \int_0^\infty \mathbf{P}_t f (x)  \dd t .
$$
Furthermore, using the semigroup property (i.e., the Markov property) at time $t/2$, and the symmetry \eqref{eq:Ptsym} of the semigroup,
$$
\langle f, \mathbf{P}_t f \rangle_{L^2 (\mu_\ph)} = \langle f, \mathbf{P}_{t/2} \mathbf{P}_{t/2} f \rangle_{L^2 (\mu_\ph)}  = \langle \mathbf{P}_{t/2}f, \mathbf{P}_{t/2} f \rangle_{L^2 (\mu_\ph)} = \|\mathbf{P}_{t/2} f\|_{L^2(\mu_\ph)}^2 \ge 0.
$$
Thus, using Fubini's theorem,
\begin{align*}
\langle f, \mathbf{G}f \rangle_{L^2(\mu_\ph)} & = \int_\Sigma \mu_\ph(\dd x) f(x)  \int_0^\infty \mathbf{P}_t f(x) \dd t\\
& = \int_0^\infty \langle f, \mathbf{P}_t f \rangle_{L^2 (\mu_\ph)} \dd t \\
& = \int_0^\infty \| \mathbf{P}_{t/2} f \|_{L^2 (\mu_\ph)}^2 \dd t \ge 0,
\end{align*}
as desired. Thus $\mathbf{G}$ is a nonnegative operator. 

Furthermore suppose that $f \in L^2 ( \mu_\ph)$ and $\mathbf{G}f = 0$. Then by the above, $\| \mathbf{P}_{t/2} f \|^2 = 0$, for Lebesgue almost every $t\ge 0$. Letting $t \to 0$ (staying away from a Lebesgue null set) and using the strong continuity of the semigroup on $L^2 ( \mu_\ph)$ we deduce $f = 0$. 
\end{proof}

Note that, as a consequence of the Riesz--Schauder theorem, all eigenvalues $\bmu_n$ of $\mathbf{G}$ are finite-dimensional, i.e., the corresponding eigenspaces have finite dimension. With this in hand we can finally define the spectrum of Liouville quantum gravity simply by taking the reciprocal of $\bmu_n$.

\begin{defn} 
The eigenvalues  set $\bl_n$ of associated to the field $\ph$ are given by $\bl_n = 1/\bmu_n$. These eigenvalues are associated to an orthonormal decomposition of $L^2 (\mu_\ph)$, namely the basis $(\ef_n)_{n\ge 1}$ of eigenfunctions. 
\end{defn}

We note immediately that $\bl_n \to \infty$. Furthermore, since $\mathbf{G}$ is Hilbert-Schmidt we deduce that $\sum_n \bl_n^{-2} < \infty$, a.s. In the Weyl law below we will get an exact asymptotic rate (in probability) of divergence of $\bl_n $ towards infinity. Note that the eigenfunctions $\ef_n$ are a.s. continuous (or more precisely there exists a.s. a continuous version of $\ef_n$) since 
\begin{equation}\label{eq:EFreg}
\ef_n = \bl_n \mathbf{G}\ef_n, \quad \mu_\ph-a.e.
\end{equation}
The continuity of the right hand side follows from the fact that $\ef_n \in L^2 ( \mu_\ph)$, together with the fact that $g_\Sigma$ blows up only logarithmically, the uniform estimate in Lemma \ref{L:modcont}, and the dominated convergence theorem. Thus $\ef_n$ admits a continuous modification which is necessarily unique. We will denote by $\ef_n$ this continuous version in what follows.

We now claim that we can without loss of generality assume that the eigenfunctions are also eigenfunctions of $\mathbf{P}_t$ for every $t\ge 0$.

\begin{lemma}\label{L:GP_t}
 We may choose the orthornormal basis $(\ef_n)_{n\ge 1}$ in such a way that for all $n\ge 1$, $\mathbf{G} \ef_n = \bmu_n \ef_n$ and for all $t \ge 0$, $\mathbf{P}_t f_n = e^{ - \bl_n t} \ef_n$.
\end{lemma}

\begin{proof} Recall that for $f \in L^2 ( \mu_\ph)$ we have $\mathbf{G} f= \int_0^\infty \mathbf{P}_t f \dd t$. Since $\mathbf{P}_t: L^2 ( \mu_\ph) \to L^2 ( \mu_\ph)$, we have $\mathbf{G} \mathbf{P_t} f = \mathbf{P}_t \mathbf{G} f$ so that $\mathbf{G}$ and $\mathbf{P}_t$ commute. Suppose $\bmu$ is an eigenvalue of $\mathbf{G}$ and let $E_{\bmu} = \ker ( \mathbf{G} - \bmu I)$ denote the corresponding eigenspace. Let $f \in E_{\bmu}$. Then 
$$
\mathbf{G} \mathbf{P_t} f = \mathbf{P}_t \mathbf{G} f = \bmu \mathbf{P}_t f. 
$$
Thus $\mathbf{P}_t : E_{\bmu} \to E_{\bmu}$ leaves this subspace invariant. Now consider the restriction $\mathbf{P}_t|_{E_{\bmu}}$ of each $\mathbf{P}_t $ to the finite-dimensional subspace $E_{\bmu}$. This is a linear operator on a finite-dimensional space, which is symmetric (because it is symmetric in the full space thanks to \eqref{eq:Ptsym}) and thus diagonalisable. As $ \mathbf{P}_t$ commutes with $\mathbf{P}_s$ for every $s,t\ge 0$ and $E_{\bmu}$ is finite-dimensional, we can find a common basis of $E_{\bmu}$ which diagonalises each $\mathbf{P}_t|_{E_{\bmu}}, t\ge 0$. Let $\ef$ be one of the elements of this basis, and let $\bl(t)$ be such that $\mathbf{P}_t \ef = \bl(t) \ef$. By the semigroup property we have $\bl(t+s) = \bl(t) \bl(s)$ for all $s \ge 0$, furthermore (by strong continuity of the semigroup) $t\mapsto \bl(t)$ must be continuous. Thus $\bl(t) = e^{ - \bl t }$ for some $\bl \in \R$.  Finally, from the fact that
$$
\bmu \ef = \mathbf{G} \ef = \int_0^\infty \mathbf{P}_t \ef \ \dd t 
$$
we deduce that $\bl>0$ and $\bmu = \bl^{-1}$. 
\end{proof}

\subsection{Heat kernel, bridge identity and heat trace formula}

The definition of the spectrum is somewhat indirect in that so far it is not completely clear how to do computations involving the eigenvalues. Our main tool to do this will be the so-called \textbf{heat trace formula} which connects the eigenvalues to the \textbf{heat kernel}, i.e., the density of the law of Liouville Brownian motion $\mathbf{Z}_t$ at time $t\ge 0$, with respect to the Liouville measure $\mu_\ph ( \dd y )$. It is a priori not entirely obvious that such a density exists, and this was first shown in \cite{GRV_hk}:

\begin{prop} There a.s. exists a family $(\mathbf{p}_t (\cdot ,\cdot))_{t\ge 0}$ of (jointly in all three arguments) measurable functions such that for all $t \ge 0$, for all $x \in \Sigma$,
 $$\mathbf{P}_t f(x) = \int_\Sigma \mathbf{p}_t (x,y) f(y) \mu_{\ph} (\dd y).$$ This family of functions $\mathbf{p}_t (x,y)$ (a priori only defined for all $t\ge 0$, for all $x\in \Sigma$ and $\mu_\ph-$almost every $y \in \Sigma$ is called the \textbf{Liouville heat kernel} and satisfies the following properties: 
\begin{enumerate}
\item (Nonnegativity) For all $t \ge 0$, for all $x \in \Sigma$ and for $\mu_\ph-$almost every $y\in \Sigma$, $\mathbf{p}_t (x,y) \ge 0$. 
\item (Symmetry) For all $t \ge 0$, for all $x,  y \in \Sigma$, $\mathbf{p}_t (x,y) = \mathbf{p}_t (y,x) $.
\item (Semigroup) For all $s, t \ge 0$, for all $x \in \Sigma$ and for $\mu_\ph-$almost every $y \in \Sigma$, $\mathbf{p}_{t+s} (x, y) = \int_\Sigma \mathbf{p}_t(x,z) \mathbf{p}_t(z,y) \mu_\ph(\dd z)$. 
\end{enumerate}
\end{prop}

See Theorem 1.4 in \cite{GRV_hk}. This is a rather soft consequence of the fact that the resolvent operator is strong Feller (just as the semigroup is) and arguments from Dirichlet form theory, more precisely Theorems 4.1.2 and 4.2.4 in \cite{Fukushima}.

\medskip Handling the heat kernel directly is in itself very difficult; there is no formula relating directly to the standard (Riemannian) heat kernel. Related to this is the fact that, at least intuitively, the heat kernel encodes in a rather direct way much of the complex, multifractal geometry of LQG: for instance, if $x \neq y$ and $t\to 0$, it seems intuitively clear that $\mathbf{p}_t(x,y)$ is carried by trajectories that stay close to geodesics between $x$ and $y$. Despite this complexity, there is an easy but remarkably useful formula, the \textbf{bridge identity} which allows us to access time integrals of the heat kernel. This identity expresses the fact that Liouville Brownian motion is a time-change of ordinary Brownian motion and can thus ultimately be seen as a consequence of conformal invariance of Brownian motion in two dimensions. It is one of the reasons why it is possible to analyse the spectral geometry of LQG further than for some other models of random geometry. 

\begin{theorem}\label{T:bridge}
Let $\psi: [0, \infty) \to [0, \infty)$ be a nonnegative Borel function. Then for all $x\in \Sigma$, for $\mu_\ph-$a.e. $y \in \Sigma$, 
\begin{equation}
\int_0^\infty \psi(t) \mathbf{p}_t (x,y)\dd t  = \int_0^\infty  \bdec{x}{y}{t}  [ \psi( F(t))] p^\Sigma_t(x,y) \dd t ,
\end{equation}
where $ \bdec{x}{y}{t}$ denotes expectation with respect to the law of an ordinary Brownian bridge $(b_s)_{0\le s \le t}$ from $x$ to $y$ of duration $t$ (independent of the realisation of the underlying field $\ph$), $F(t) = \lim_{\eps \to 0} \eps^{\gamma^2/2}\int_0^t e^{\gamma \ph_\eps( b_s)} \dd s$ is the quantum clock appearing in the definition of Liouville Brownian motion, and $p_t^\Sigma(x,y)$ is the ordinary transition probabilities of Brownian motion killed when it leaves $\Sigma$ (thus $p_t^\Sigma(x,\cdot)$ is a subprobablity kernel). 
\end{theorem}

The proof is an easy consequence of the definition of Liouville Brownian motion as a time change of ordinary Brownian motion, and the definition of the Liouville heat kernel. See \cite{RVspectral, HKPZ, MRVZ, BW}. Some continuity properties of the integrated heat kernel can easily be deduced from this. 

\medskip The main result of this section is the following spectral decomposition of the heat kernel, which will lead us to the desired trace formula. 

\begin{theorem}[Spectral decomposition of Liouville heat kernel]\label{T:spectraldec}
For all $t\ge 0$, all $x \in \Sigma$, almost surely and for $\mu_\ph-$almost every $y \in \Sigma$ we have
\begin{equation}\label{eq:spectraldec}
\mathbf{p}_t (x,y) = \sum_{n=1}^\infty e^{-\bl_n t} \ef_n(x) \ef_n(y). 
\end{equation}
The series on the right hand side converges a.s. uniformly on $(\eps, \infty) \times \Sigma \times \Sigma$ for any $\eps>0$. As a result, there is a (necessarily unique) \emph{jointly continuous} version of the Liouville heat kernel $\mathbf{p}_t(x,y)$. 
\end{theorem}

See (5.10) in \cite{AndresKajino} and Theorem 3.3 in \cite{MRVZ} (although the mode of convergence of the series is not discussed there, and continuity of the left hand side appears to be taken for granted). 

\begin{proof}
Let $f \in L^2 ( \mu_\ph)$. Let us write $f = \sum_{n=1}^\infty c_n \ef_n$, with $c_n = \langle f, \ef_n\rangle_{L^2(\mu_\ph)}$. Then from Lemma \ref{L:GP_t} we see that 
\begin{align*}
\mathbf{P}_t f (x) & = \sum_{n=1}^\infty c_n \mathbf{P}_t \ef_n (x) \\
& = \sum_{n=1}^\infty c_n e^{- \bl_n t} \ef_n(x) \\
& = \int_\Sigma \left(\sum_{n=1}^\infty e^{ - \bl_n t} \ef_n(x) \ef_n(y) \right) f(y) \mu_\ph(\dd y) .
\end{align*}
We conclude \eqref{eq:spectraldec} thanks to the almost everywhere uniqueness of the density of the law of $\mathbf{Z}_t$ with respect to $\mu_\ph$. 

The almost sure uniform convergence of the series on $(\eps, \infty) \times \Sigma \times \Sigma$ can be seen as a consequence of ultracontractivity of the semigroup (Proposition 5.3 in \cite{AndresKajino}) and some general results in Dirichlet form theory (Theorem 2.1.4 in \cite{Davies}). 
\end{proof}

For us, the most interesting consequence of the above spectral decomposition arises from letting $x = y$ and integrating over $x$ (letting $x=y$ is now meaningful since we work with the continuous modification of $\mathbf{p}_t (x,y)$ provided by Theorem \ref{T:spectraldec}). The resulting formula is called the \textbf{heat trace formula}. 

\begin{corollary}\label{C:heattrace}
We have for every $t>0$, 
\begin{equation}\label{eq:heattrace}
\int_\Sigma \mathbf{p}_t(x,x) \mu_\ph(\mathrm{d}x) = \sum_{n=1}^\infty e^{- \bl_n t}
\end{equation}
\end{corollary}

The quantity on the left hand side of \eqref{eq:heattrace}, $\mathbf{H}(t) = \int_\Sigma \mathbf{p}_t(x,x) \mu_\ph(\mathrm{d} x)$, is called the \textbf{heat trace}. If we think of heat kernel $\mathbf{p}_t(x,y)$ as a matrix, then $\mathbf{H}(t)$ is indeed its trace. On the right hand side of \eqref{eq:heattrace}, the series $\sum_n e^{-\bl_n t}$ can be rewritten as $\int_0^\infty e^{- \lambda t} \dd\mathbf{N}(\lambda)$, where $\mathbf{N}(\lambda) = \sum_n 1_{\{ \bl_n \le \lambda \}}$ is the eigenvalue counting function, and $\dd\mathbf{N}$ is its associated Stieltjes measure. In other words, the heat trace formula reads 
\begin{equation}\label{eq:heattrace2}
\mathbf{H}(t) = \int_0^\infty e^{- \lambda t} \dd\mathbf{N}(\lambda)
\end{equation}  
for every $t>0$. It therefore shows that the heat trace equals the Laplace transform of $\dd\mathbf{N}$. This formula is thus an extremely powerful tool, in that it relates the large $\lambda$ asymptotics to the short term behaviour of the heat trace (i.e., to the on-diagonal behaviour of the heat kernel).

\section{Weyl's law in LQG}

\subsection{Classical Weyl law}
\label{SS:Weyl_classical}
Questions concerning the spectrum of the Laplacian (especially its eigenvalues) in a given bounded domain $D$ of $\R^d$ seem to be as old as spectral theory itself. In 1882, the famous German-born, British physicist Sir Arthur Schuster (known in particular for work on spectroscopy) wrote that \emph{``it would baﬄe the most skillful mathematician to find out the
shape of a bell by means of the sound which it is capable of sending out’’}. The question seems however to have entered the mathematical stream of consciousness via a conference given in 1910 by another famous physicist, Hendrik Lorentz, in Göttingen. During the conference, which was attended by the luminaries of the time (including Felix Klein, Hermann Minkowski, and especially David Hilbert and his student Hermann Weyl), Lorentz offered the following problem:

\emph{In an encolosure with a perfectly reflecting surface there can form standing
electromagnetic waves, analogous to tones of an organ pipe. We shall confine our
attention to very high overtones\footnote{That is, to high frequencies or large eigenvalues.}. […] There arises the mathematical problem to
prove that the number of overtones which lie between frequencies $[\nu, \nu + \mathrm{d} \nu]$
 is
independent of the shape of the enclosure, and is simply proportional to its
volume… It has been verified for many simple shapes… There is no doubt that the
theorem holds in general...’'}

Thus Lorentz conjectured what later became known as the Weyl law (in two dimensions): the approximate density of eigenvalues is constant, and is proportional to the volume of the corresponding domain. Put it another way, the eigenvalue counting function $N(\lambda) = \sum_{n=1}^\infty 1_{\{\lambda_n \le \lambda\}}$ grows linearly, proportionally to the volume. It is reported that the problem very much captured the imagination of Hilbert, who (apparently) declared that this problem would not be solved in his lifetime! It must therefore have come as a shock to him when, less than two years later, his own student Hermann Weyl proved his celebrated law. We state it below using the probabilistic conventions of working with $(-1/2) \Delta$ instead of the Laplacian, and in two dimensions in order to help facilitate the comparison with the LQG case later on. 

Let $\Sigma \subset \R^2$ be a bounded domain, and let $(f_n)_{n\ge 0}$ be an orthonormal basis of eigenfunctions of $(-1/2) \Delta$ in $L^2 (\Sigma,\mathrm{d} x)$ with Dirichlet boundary conditions\footnote{These are constructed in a way that is analogous to the work done in the previous section, by viewing the standard Green function as an operator on $L^2 (\Sigma, \mathrm{d}x)$ and applying the spectral theorem to it.} It is not hard to see (arguing similarly to \eqref{eq:EFreg}) that these eigenfunctions are smooth, i.e., $f_n \in C^2(\bar \Sigma)$. To summarise,
$$
\begin{cases} 
 -\tfrac12 \Delta f_n = \lambda_n f_n & \text{ in } \Sigma\\
 f_n = 0 & \text{ on } \partial \Sigma,
\end{cases}
$$
with $\lambda_n \to \infty$.

With these notations, the classical Weyl law \cite{Weyl} takes the following form: 

\begin{theorem}\label{T:Weyl_classical}
Let $N(\lambda) = \sum_{n=0}^\infty 1_{\{\lambda_n \le \lambda\}}$ denote the eigenvalue counting function. Then 
\begin{equation}\label{eq:Weyl_classical}
\frac{N(\lambda)}{\lambda} \to c_0 |\Sigma|
\end{equation}
as $\lambda \to \infty$. Here $c_0 = 1/ (2\pi)$ and $|\Sigma|$ denotes the Lebesgue measure of $\Sigma$. 
\end{theorem}

\begin{remark} Let us make a few comments on this result, which in reality is considerably more general than what is quoted above:
\label{R:Weyl_classical}
\begin{enumerate}

\item The result is not restricted to the case of dimension two, and was in Weyl's original paper \cite{Weyl} written for dimensions two and three. If $\Sigma \subset \R^d$ is a bounded domain of $\R^d$, then $N(\lambda)$ grows asymptotically like $\lambda^{d/2}$, and the corresponding constant of proportionality is as in \eqref{eq:Weyl_classical}, with the value $c_0$ replaced by $c_0 (d) = 2 (2\pi)^{-d} \omega_d$, where $\omega_d$ is the volume of the unit ball in $\R^{d}$. 

\item We stated the result here for Dirichlet eigenfunctions, but the same result holds for Neumann eigenfunctions (taking care of first removing constant functions from consideration, since these would otherwise make the Laplacian non-invertible and so the corresponding Green function ill-defined). This Neumann setup corresponds in fact more precisely to the setting of the original problem asked by Lorentz.   

\item The Weyl law also holds not just in $\R^2$ but more generally on an arbitrary bounded Riemannian manifold $(\Sigma, g)$ (with or without boundary; in the latter case one should also first remove the constant functions, as in the Neumann case above). It might seem surprising at first sight that the constant of proportionality (i.e., the value of the limit in Theorem \ref{T:Weyl_classical}) is in fact independent of the metric $g$ on the manifold, except for its dependence on the total volume. That is,  \eqref{eq:Weyl_classical} holds where in the right hand side, $|\Sigma|$ denotes the total Riemannian volume of $\Sigma$ and $c_0 = 1/(2\pi)$ is the same whatever the metric $g$ on $\Sigma$. 

\item We stated the Weyl law for the Laplacian, but generalisations to other operators (e.g., Schrödinger operators) are also known. 

\end{enumerate}

\end{remark}

There exist a great many proofs of this formula. Weyl's original proof consisted in making use of the ``bracketing method'', cutting the shape $\Sigma$ in many small squares and using the explicitly known eigenfunctions in each small square to construct eigenfunctions on a square-grid like approximation of the entire domain $\Sigma$; the variational characterisation of eigenvalues via Rayleigh's quotient ensures monotonicity of the eigenvalues with respect to the domain, which essentially concludes the proof. For a probabilist, the most efficient proofs today are in fact much quicker: as we will see (and soon make use of), one can use the trace formula (the deterministic analogue of \eqref{eq:heattrace2}) to relate the large $\lambda$ asymptotics of $N(\lambda)$ (more precisely, its associated Stieltjes measure) to the short term behaviour of the heat kernel. This heat kernel is that of Brownian motion killed on the boundary of $\Sigma$. But for most points $x\in \Sigma$, when  $t\to 0$, the fact that Brownian motion is killed when leaving $\Sigma$ is irrelevant, since the boundary is ``far'', relative to the point $x$. One can therefore guess (and it is rather simple to jusify) that 
\begin{equation}
p^\Sigma_t(x,x) \sim \tfrac{1}{2\pi t}
\end{equation}
as $t \to 0$, where the right hand side comes from the on-diagonal transitions of Brownian motion on the whole plane, i.e., the density at the origin of the two-dimensional centered Gaussian probability distribution with covariance matrix $tI$. This gives the asymptotics of the heat trace: 
\begin{equation}\label{eq:heattrace_classical}
H(t) : = \int_\Sigma p^\Sigma_t(x,x) \dd x \sim \frac1{2\pi t} |\Sigma|. 
\end{equation}
To transfer this result onto the asymptotics of $N(\lambda)$ it remains to apply a Tauberian theorem such as Karamata's theorem (a probablistic version of which is given below in Lemma \ref{L:tauberian}; here it suffices to apply the deterministic version with $\rho =1$). 

\paragraph{Beyond Weyl's law.} The above result can be regarded as the foundational stone of an entire domain of mathematics (spectral geometry) and has given rise to a huge number of works in many directions, see for instance \cite{Ivrii100} for some discussions. We point out a few important directions here, and will discuss some more later in the LQG setup. 

\begin{enumerate}

\item Partly inspired by Weyl's work (but seemingly unaware that the question had already been asked by Schuster) Mark Kac coined the famous question ``\emph{Can one hear the shape of a drum?}'' in his 1966 survey article \cite{Kac}. In this description (which is actually slightly reductive compared to the real-world problem), the sounds one can hear from a drum being struck are its eigenvalues. Thus the question is: does the spectrum $\{\lambda_n\}_{n\ge 1}$ of $(-1/2) \Delta$ in $\Sigma$ with Dirichlet boundary conditions actually determine $\Sigma$ uniquely, up to isometries of $\R^d$? From Weyl's law, the spectrum at least determines the volume $|\Sigma|$ of $\Sigma$, so one can at least ``hear the volume'' of the drum. Shortly before Kac's article (and this was probably the main inspiration for Kac's work), the topologist John Milnor \cite{Milnor} observed that there exists a pair of flat tori in 16 dimensions  which are isospectral (i.e., have the same eigenvalue sequence) but are not isometric. However it was not until 1992 that  Gordon, Webb and Wolpert \cite{Iso} found two isospectral but non-isometric bounded planar domains of $\R^2$. This resolves in the negative Kac's original question.

\item Since the Weyl law provides the clearest \emph{spectral invariant} (observable of $\Sigma$ determined by the eigenvalue sequence), it is natural to ask more precise about the asymptotic behaviour of $N(\lambda)$ as $\lambda \to \infty$, since (for the same reasons) any additional terms provides further spectral invariants. In his original paper, Weyl \cite{Weyl} actually made a precise conjecture about the next term for the eigenvalue counting function: in $\Sigma \subset \R^d$ with $d\ge 1$, the conjecture is that 
\begin{equation}\label{eq:Weyl_next}
N(\lambda) = c_0 \lambda^{d/2} |\Sigma| - c'_0 \lambda^{(d-1)/2} |\partial \Sigma| + o(\lambda^{(d-1)/2}),
\end{equation}
where $c'_0 = c'_0(d)$ is another explicit constant, and $|\partial \Sigma|$ is the length of the boundary (when that length is well-defined, e.g. if $\partial \Sigma$ is smooth). In particular, the boundary length should be a spectral invariant too. (In the Neumann case, the same result is believed to hold, except for the sign of the correction which is reversed). Amazingly, this conjecture is still open in general! However, we know of at least two results which strongly support this conjecture. First, Courant \cite{Courant} proved in 1920 a bound on the discrepancy between the eigenvalue counting function $N(\lambda)$ and its Weyl limit:
\begin{equation}\label{eq:Courant}
\big|N(\lambda) - c_0 \lambda^{d/2} |\Sigma|\big| = o ( \lambda^{(d-1)/2} \log \lambda). 
\end{equation}
Secondly, Ivrii \cite{Ivrii} proved in 1980 that Weyl's conjecture \eqref{eq:Weyl_next} is true, if one makes the following additional assumption on $\Sigma$: the measure of all closed (periodic) geodesics in $\Sigma$ is zero (in other words, there are not too many such geodesics). Ivrii further conjectured that every smooth, bounded domain in $\R^d$ satisfies this property; unfortunately to this day this remains an open problem. See \cite{Ivrii100} for a survey of some of these developments. 

\item While it is hard to obtain an expansion at infinity of the eigenvalue counting function beyond the first term (as discussed above), it is on the other hand considerably easier to obtain short time asymptotics of the heat trace $H(t) = \int_\Sigma p^\Sigma_t(x,x) \dd x$. The first term is, as discussed above in \eqref{eq:heattrace_classical}, proportional to the volume and $(1/t)$, and this corresponds to the Weyl law. The next term is also known and ``corresponds'' to the conjecture \eqref{eq:Weyl_next} about the next term in the Weyl law, namely
\begin{equation}\label{eq:heattrace_classical2}
H(t) = \frac{1}{2\pi t} |\Sigma| - \frac{c'_0}{\sqrt{t}} |\partial \Sigma| + o( 1/ \sqrt{t}),
\end{equation}
and $c'_0 = 1/ (8 \sqrt{\pi/2})$. In particular, this shows that the boundary length \emph{is} spectrally determined, since $H(t)$ is the Laplace transform of the eigenvalue counting function and is thus spectrally determined too. While ``morally'' this expansion is analogous to that in \eqref{eq:Weyl_next}, and despite the relation \eqref{eq:heattrace2} between heat trace and eigenvalues, there is no direct way to transfer the result from the heat trace to the eigenvalue counting function when it comes to second order asymptotics. This is because Tauberian theory requires an assumption of monotonicity and so would not apply to, e.g., the Stieltjes measure associated to $N(\lambda) - c_0 \lambda |\Sigma|$, say.  

\item In fact, much more precise asymptotics than \eqref{eq:heattrace_classical2} are known. In particular, subsequent terms (assuming smoothness of the boundary $\partial \Sigma$) involve the so-called \textbf{$\zeta$-regularised determinant of the Laplacian}, itself a quantity of considerable interest in the context of Liouville theory (as it is in fact equal, up to a constant to the classical Liouville action which underpins the Polyakov action \eqref{E:Polyakov}), and thus satisfies a so-called \textbf{Polyakov--Alvarez} conformal anomaly formula. See e.g. \cite{TakhtajanZograf} and \cite{BOW} for details. Likewise, if $\partial \Sigma$ is smooth except at a finite number corners of given angle, the impact on the heat trace of these corners is precisely known as a function of the number and angle of these corners, see \cite{LuRowlett, Krusell}. In particular, both the number of corners and their angles of the shape $\Sigma$, as well as the value of its $\zeta$-regularised determinant of the Laplacian, are spectral invariants. 
\end{enumerate}

\subsection{Statement of the result and sketch of proof}

We now have the tools in hand to state and discuss the result obtained in \cite{BW} for the Weyl law in LQG. We let $\Sigma$ be a bounded, simply connected domain. (Note that we do not assume any smoothness of the boundary beyond simple connectedness; in fact even the latter can probably be removed by assuming that there is at least point $z_0 \in \partial \Sigma$ on the boundary which is regular for the Green function: that is, $P_{z_0}(\inf \{t > 0: W_t \not \in \Sigma\} = 0) = 1$, i.e., $W$ leaves $\Sigma$ immediately, where $W$ is a Brownian motion. It is well known that such an assumption implies that the Green function $g_\Sigma(x,y)$ in $\Sigma$ with Dirichlet boundary condition is everywhere finite away from the diagonal.) Let us fix a coupling constant $0< \gamma<2$, and let $\mathbf{N}(\lambda)$ denote the LQG eigenvalue counting function as in Section \ref{SS:Green_spectrum}, associated to an underlying Gaussian free field $\ph$ with Dirichlet boundary conditions on $\Sigma$ and associated Liouville volume measure (or Gaussian multiplicative chaos measure) $\mu_\ph$. The main result of \cite{BW} is the following Weyl law in LQG. 

\begin{thm}\label{T:Weyl} Let $0< \gamma <2$.
We have
\begin{align}\label{eq:Weyl}
\frac{\mathbf{N}_{\gamma}(\lambda)}{\lambda} \xrightarrow[\lambda \to \infty]{p} c_\gamma \mu_{\ph}(\Sigma),
\end{align}
where $ \xrightarrow{p}$ denotes convergence in probability, and the constant $c_\gamma$ satisfies
\begin{equation}
c_\gamma = \frac1{\pi(2- \gamma^2/2)}.
\end{equation}
\end{thm}

In \cite{BW}, the constant $c_\gamma$ was identified in two steps. First, the result was obtained with $c_\gamma$ being a certain explicit but complicated functional of the so-called $\gamma$-quantum cone of Sheffield \cite{zipper}. Then, using an indirect calculation (i.e., \emph{not} using the explicit expression in terms of the quantum cone) its value was computed and identified explicitly. A more transparent interpretation of $c_\gamma$ is given in \cite{BK} and will be described below (see \eqref{eq:cgamma_cone}). 

\begin{remark} We make several remarks about this result. 

\begin{enumerate}
\item Comparing with the $d$-dimensional form of the classical Weyl law, the statement is (partly) consistent with a ``two-dimensional'' Weyl behaviour: namely the growth of $\mathbf{N}(\lambda)$ is proportional to $\lambda^{d/2}$ with $ d= 2$. The power $d$ appearing in the growth of the eigenvalue counting function as here is (closely related to) a quantity known as the spectral dimension. Thus here the spectral dimension of LQG is two, for any $0<\gamma<2$. Such a behaviour had been predicted by Ambj\o rn et al. \cite{Ambjorn}, and a rigorous proof of a closely related fact had first been provided by Rhodes and Vargas \cite{RVspectral}.

\item However the constant of proportionality $c_\gamma$ does not match the Weyl constant $c_0 = 1/(2\pi)$ from Theorem \ref{T:Weyl_classical}, except of course if the coupling constant $\gamma = 0$. This is all the more surprising in view of the fact that, as mentioned in Remark \ref{R:Weyl_classical}, the constant $c_0$ is the same on \emph{any} Riemannian manifold, independently of the metric. Note in particular that, although it might at first sight be tempting to relate the appearance of the factor $(2-\gamma^2/2)$ in the denominator of $c_\gamma$ to the Hausdorff dimension of the set of so-called $\gamma$-thick points (known to support $\mu_\ph$, see, e.g., \cite{BP}) there does not appear to be any reason why this fact should be at all relevant here. This can be seen from several different points of view. First, changing the dimension of the ``support'' of $\mu_\ph$ might reasonably be thought to affect the spectral dimension, not the constant of proportionality. Secondly, and more importantly, the Hausdorff dimension of the set of thick points is a Euclidean measure of the size of the set of thick points, and is thus not intrinsic to LQG (in contrast to the question studied in Theorem \ref{T:Weyl}). 

\end{enumerate}  

\end{remark}

Let us now sketch the main steps of the proof of this result. 

\begin{proof}[Sketch of proof of Theorem \ref{T:Weyl}] In common with the proof outlined for the classical Weyl law, the starting point of the proof is to consider the heat trace $\mathbf{H}(t) = \int_\Sigma \mathbf{p}_t(x,x) \mu_\ph(\dd x)$. While this can be expected to hold in many models of random geometry, usually such a formula is not directly useful because it is hard to get a handle on the heat kernel. (For instance, in randomised Sierpinski gaskets and other random fractals such as continuous random trees and stable trees, for which Weyl-type asymptotics are known -- see \cite{Hambly, CH2008, CH2010} -- the trace formula can only be used to derive \emph{a posteriori} estimates on the heat kernel, rather than the other way around). 
 However, in the case of LQG we can get a handle on the heat kernel via the bridge formula of Theorem \ref{T:bridge}, which gives a concrete expression for (time integrals of) the heat kernel and which, fundamentally, expresses the conformally invariant nature of the problem. The final ingredient is a certain ``self-similarity'' of the so-called quantum cone, as will be explained below. 

\paragraph{Step 1.} The starting point is to prove a probabilistic extension of Karamata's classic Tauberian theorem. This extension is required because of the mode of convergence in Theorem \ref{T:Weyl}, which is that of convergence in probability rather than almost sure. (For almost sure convergence, the theorem itself can of course be directly applied and no extension is required). 

Recall that a function $L: (0, \infty) \to (0, \infty)$ is slowly varying at zero if $\lim_{t \to 0^+} L(xt)/L(t) = 1$ for any $x > 0$ (and analogously for slow variation at infinity).
\begin{lemma}[Theorem A.2 in \cite{BW}]\label{L:tauberian}
Let $\nu$ be a nonnegative random measure on $\mathbb{R}_+$, $\nu(t) := \int_0^t \nu(ds)$, and suppose the Laplace transform
\begin{align*}
    \hat{\nu}(\lambda) := \int_0^\infty e^{-\lambda s} \nu(ds)
\end{align*}

\noindent exists almost surely for any $\lambda > 0$. If
\begin{itemize}\setlength\itemsep{0em}
\item $\rho \in [0, \infty)$ is fixed;
\item $L: (0, \infty) \to (0, \infty)$ is a deterministic slowly varying function at $0$; and 
\item $C_\nu$ is some non-negative (finite) random variable,
\end{itemize}

\noindent then we have:
\begin{align}\label{eq:tauberian}
    \frac{\lambda^\rho}{L(\lambda^{-1})} \hat{\nu}(\lambda) \xrightarrow[\lambda \to \infty]{p} C_\nu
    \qquad \Rightarrow \qquad 
    \frac{t^{-\rho}}{L(t)} \nu(t) \xrightarrow[t \to 0^+]{p} \frac{C_\nu}{\Gamma(1+\rho)}.
\end{align}

\noindent The same implication also holds when one considers the asymptotics as $\lambda \to 0^+$ and $t \to \infty$ in \eqref{eq:tauberian} (but with $L$ being slowly varying at infinity) instead.
\end{lemma}

 By Lemma \ref{L:tauberian}, and in view of the heat trace formula \eqref{eq:heattrace2}, in order to prove Theorem \ref{T:Weyl} it suffices to establish $t \mathbf{H}(t) \to c_\gamma \mu_\ph(\Sigma)$ in probability as $t\to 0$. In fact, we will establish the following stronger statement: if $A\subset \Sigma$ is any open set, and $\mathbf{H}_A(t) = \int_A \mathbf{p}_t(x,x) \mu_\ph(\dd x)$, then
\begin{equation}\label{eq:goalheattrace3}
\mathbf{H}_A(t)  \sim \frac{c_\gamma \mu_\ph(A)}{t},
\end{equation}
in the sense that the ratio of the two sides tends to one in probability as $t\to 0$. Unfortunately however, we do not have a direct access to $\mathbf{p}_t(x,x)$. Ultimately, our only hope is to integrate over time and make use of the bridge formula (Theorem \ref{T:bridge}), since time integrals of the Liouville heat kernel can be expressed in terms of time integrals of the standard heat kernel, with an extra multiplicative factor accounting for the time-change defining Liouville Brownian motion. For instance, a natural approach is to consider $\int_\eps^1 \mathbf{H}_A(t) \dd t$, and try to show this behaves like 
\begin{equation}\label{eq:wrong}
\int_\eps^1 \mathbf{H}_A(t) \dd t \sim c_\gamma \mu_\ph ( A) \log (1/\eps),
\end{equation}
in probability as $\eps \to 0$ (i.e., the ratio of the two sides tends to one as $\eps \to 0$ in probability). However, even taking into account the monotonicity $t\mapsto \mathbf{H}_A(t)$, \eqref{eq:wrong} alone is not sufficient to ensure \eqref{eq:goalheattrace3}. An instructive example is to consider 
$$
H(t)  = \frac{1 + a \sin(\log t)}{ t}$$  for $t\in (0,1)$. Then $H'(t) = -t^{-2} [1 - a \cos (\log t)]$, so if $0<a<1$ (say $a = 1/2$) we have that $H'(t) <0$ and $H$ is monotone in $t$, just like $t\mapsto\mathbf{H}_A(t)$ is. Furthermore,  $\int_\eps^1 H(s) \dd s = -\log \eps - a[1 - \sin(\log \eps)] = -\log \eps + O(1)$. So the analogue of \eqref{eq:wrong} holds for this example, yet $H(t)$ itself is not asymptotic to $1/t$ when $t\to 0$. This counterexample is related to the limitations of Tauberian theory, for which the exponent $\rho =-1$ is critical and falls outside its scope. (A more precise theory exists which deals with this critical case, known as de Haan's theory, but its assumptions are far too stringent for our purpose).   
 
Instead of \eqref{eq:goalheattrace3}, we claim it suffices to show that there exists some arbitrary $\alpha>0$ (we will take $\alpha = 1$) such that
\begin{equation}
\label{eq:heattrace3int}
 \int_0^\eps t^\alpha \mathbf{H}_A(t) \dd t  \sim c_\gamma \mu_\ph(A) \frac{\eps^\alpha}{\alpha}
\end{equation}
as $\eps \to 0$, again in the sense that the ratio of the two sides tends to 1 in probability. That \eqref{eq:heattrace3int} indeed really implies \eqref{eq:goalheattrace3} is provided by a Tauberian-type argument, which exploits the monotonicity in $t$ of $\mathbf{H}_A(t)$ (see Lemma A.1 in \cite{BW}).

\paragraph{Step 2.} We now seek to establish \eqref{eq:heattrace3int} with $\alpha = 1$. Using once more Lemma \ref{L:tauberian}, by monotonicity of $t\mapsto \mathbf{H}_A(t)$, this boils down to showing that 
\begin{equation}
\int_0^\infty e^{-\lambda t} t \mathbf{H}_A(t) \dd t \sim \frac{c_\gamma \mu_\ph (A)}{\lambda} 
\end{equation}
 in probability as $\lambda \to \infty$. But using the bridge formula of Theorem \ref{T:bridge}, it suffices in fact to show 
 \begin{equation}
\label{eq:goalbridge}
\int_A \mu_{\ph}(\dd x) \int_0^\infty \bdec{x}{x}{u}\left[F(u) e^{-\lambda F(u)}1_{\{ \tau_\Sigma>u \}}\right] p^\Sigma_u(x, x) \dd u \sim \frac{c_\gamma \mu_\ph (A)}{\lambda} ,
 \end{equation}
where, as in \eqref{eq:clock}, $F(u) = \int_0^u e^{\gamma \ph(b_s)} \dd s $ is the time-change defining Liouville Brownian motion (applied here to a Brownian bridge $(b_s)_{0\le s \le u}$ of duration $u$ instead of Brownian motion). Informally, one can think of $\lambda>0$ here as a killing rate. When $\lambda \to \infty$ this has the effect of restricting the duration of the corresponding trajectory such that $F(t) \lesssim 1/ \lambda$ (but one should pay attention to the fact that a Brownian bridge time-changed by $F$ is not the same as a Liouville Brownian motion from $x$ to $x$). 

At a high level, the proof of \eqref{eq:goalbridge} is inspired by the ``elementary approach to Gaussian multiplicative chaos'' from \cite{BerestyckiGMC}: that is, we show convergence in $L^2$ of a (truncated) version of $\lambda$ times the left hand side of \eqref{eq:goalbridge}, (i.e., $\lambda \int_A \mu_{\ph}(\dd x) \int_0^\infty \bdec{x}{x}{u}\left[F(u) e^{-\lambda F(u)}1_{\{ \tau_\Sigma>u\}}\right] p^\Sigma_u(x, x) \dd u$), towards $c_\gamma \mu_\ph(A)$. The truncation is related to the notion of thick points of $\ph$ (see, e.g., \cite{BP}): points which are of a thickness greater than $\alpha$ (with $\alpha>\gamma$ given and chosen close to $\gamma$) do not contribute much to the expectation but make the second moment blow up, so they must first be removed. The calculations involved with the computation of these moments are quite involved and we refer to \cite{BW} for details. However, in order to illustrate the main idea, we explain how to at least compute the relevant expectation (one point function): 
\begin{lemma} \label{L:1pt}
\begin{equation}\label{eq:1pt}
\lim_{\lambda \to \infty} \E \left[\lambda \int_A \mu_{\ph}(\dd x) \int_0^\infty \bdec{x}{x}{u}\left[F(u) e^{-\lambda F(u)}1_{\{ \tau_\Sigma>u\}}\right] p^\Sigma_u(x, x) \dd u\right] = c \E[ \mu_\ph(A)],
\end{equation}
for some non-explicit $c>0$ (according to \eqref{eq:goalbridge} we expect $c = c_\gamma = 1/ [\pi(2- \gamma^2/2)]$). 
\end{lemma}
\begin{proof}[Sketch of proof of Lemma \ref{L:1pt}] By Girsanov's lemma (see \cite[Theorem 3.16]{BP}), the expectation in the left hand side of \eqref{eq:1pt} can be written as 
$$
\int_A R(x, \Sigma)^{\gamma^2/2} \E^*\left[ \int_0^\infty \bdec{x}{x}{u} \big[\cI ( \lambda F(u) ) 1_{\{ \tau_\Sigma>u\}} p^\Sigma_u(x,x) \dd u\big] \right] \dd x 
$$
where:
\begin{itemize}

\item $R(x, \Sigma)$ is the conformal radius of $\Sigma$ viewed from $x$ (thus, since $\ph$ is a GFF on $\Sigma$ with Dirichlet boundary conditions, $ \E[ \mu_\ph(\dd x) ] =R(x, \Sigma)^{\gamma^2/2} \dd x$, see \cite[Theorem 2.1]{BP});

\item  $\E^*$ denotes the ``measure rooted at $x$'': that is, $\ph$ has the law of a GFF on $\Sigma$ with Dirichlet boundary conditions, plus a function given by $\gamma G_\Sigma (x, \cdot)$; 

\item The function $\cI$ is defined by setting $\cI(s) = se^{-s}$ for any $s \ge 0$.
\end{itemize}
We approximate $p^\Sigma_u(x,x) $ by $1/(2\pi u)$, which is a good approximation if $u$ is small and $x$ is not close to the boundary of $\Sigma$. A Brownian bridge of duration $u$ has a range of order $\sqrt{u}$. However, we know from basic properties of the Gaussian free field that it is a good idea to re-parametrise the relevant spatial scale as an exponential function of some parameter interpreted as ``time'' (see, e.g., \cite[Theorem 1.59]{BP} which shows that circle average evolve like a Brownian motion on a logarithmic scale). We thus make the change of variables $u = e^{-2t}$ and thus, noting that $\dd t \propto \dd u / u$, it suffices to show that
$$
\lim_{\lambda \to \infty}  \E^* \int_0^\infty \bdec{x}{x}{e^{-2t}} \big[\cI ( \lambda F(e^{-2t}) ) 1_{\{ \tau_\Sigma> e^{-2t}\}}  \dd t\big]   = C  
$$
for some constant $C>0$. So far we assumed that $\ph$ was a GFF with Dirichlet boundary conditions in $\Sigma$. However, it is easier for the following computation to pretend that $\Sigma = \mathbb{D}$ is the unit disc, $x= 0$, and $\ph$ is an \emph{exactly scale-invariant field}, see \cite[Section 3.7]{BP}. (One could also do this argument with a so-called $\gamma$-quantum cone). In such a field, by definition, for any $0<r<1$, 
$$
(\ph_{r\eps} ( r z) )_{z\in B(0,1)} = (\Omega_r + \ph_\eps(z))_{z\in B(0,1)}
$$ 
where in the right hand side, $\Omega_r$ is a centered Gaussian random variable with variance $\log (1/r)$ which is independent of $\ph$, see \cite[Corollary 3.22]{BP}. Applying this with $r = e^{-t}$ (which is the relevant spatial scale), and writing $B_t$ instead of $\Omega_r$ (which, roughly speaking, corresponds to the circle average of $\ph$ at that scale), and exploiting also the scale-invariance of a Brownian bridge, we see that it suffices to show
\begin{equation}\label{eq:I}
\lim_{\lambda \to \infty}  \E^* \int_0^\infty \bdec{0}{0}{1} \big[\cI ( \lambda e^{\gamma (B_t  - (Q- \gamma) t)} F(1) ) 1_{\{ \tau_{e^{t} \mathbb{D}}> 1\}}  \dd t\big]   = C ,
\end{equation}
where $C>0$ is a constant and $Q = \tfrac{2}{\gamma} + \tfrac{\gamma}{2}$. Here $B_t$ is independent of $F(1)$ and is a centered Gaussian variable of variance $t$ and may thus be thought of as the marginal at time $t>0$ of a standard (one-dimensional) Brownian motion. The appearance of Brownian motion with drift $\gamma -Q$ here is not surprising in view of the change of coordinates formula of LQG (Theorem 2.8 in \cite{BP}; this also describes, essentially for the same reasons, the radial component of the $\gamma$-quantum cone which governs the local limit of the geometry of the field $\ph$ around the point $x$ under $\E^*$). 

Now let us recall that since $\cI(s) = se^{-s}$ we have $\cI(s) \to 0$ as $s\to \infty$. Thus if $t >0$ is fixed and $\lambda\to \infty$, the argument in the function $\cI$ of \eqref{eq:I} converges to 0. However it does not follow that the entire integral tends to 0: as $t\to \infty$, $B_t - (Q- \gamma) t \to - \infty$, so one can find $t$ very large so that $e^{\gamma ( B_t - (Q-\gamma) t)}$ compensates $\lambda F(1)$. For such values of $t$, the indicator function $1_{\{ \tau_{e^{t} \mathbb{D}}> 1\}} $ is essentially equal to one, so we will ignore this from now on. The above discussion suggests naturally to introduce the stopping time
$$\tau = \inf \{ t >0: e^{\gamma R_t} = \frac{1}{ \lambda F(1)} \}, \text{ where } R_t = B_t - (Q- \gamma) t \text{ is the radial component}.$$
Crucially, since the Brownian motion $B$ is independent from $F(1)$, then it is easy to describe the law of $R$ around time $\tau$: by William's path decomposition theorem \cite{Wil1974}, $(R_{\tau + t })_{t\in [-L, \infty)} = (\hat R_t)_{t\in [-L, \infty)}$ in distribution, where 
$$
 \hat R_t= 
\begin{cases}
B_t - (Q- \gamma)t & \text{ if } t \ge 0\\
\hat B_{|t|} + (Q- \gamma) |t| & \text{ if } t <0.
\end{cases}
$$
with $B$ a Brownian motion, $\hat B$ an independent Brownian motion conditioned so that $\hat B_s + (Q- \gamma) s \ge 0$ for all $s\ge 0$. The range of validity of this identity in distribution is the interval $[-L, \infty)$ where 
$$
L : = \sup \{ s>0: \hat B_s + (Q- \gamma) s = \lambda F(1)\}.
$$  
As $\lambda \to \infty$, this $L$ disappears to infinity, and thus we find 
$$
\E^* \int_0^\infty \bdec{0}{0}{1} \big[\cI ( \lambda e^{\gamma (B_t  - (Q - \gamma)t)} F(1) ) 1_{\{ \tau_{e^{t} \mathbb{D}}> 1\}}  \dd t\big] \to \E^* \int_{-\infty}^\infty \cI\left( \big (e^{\gamma \hat R_t} )\right) \dd t .
$$
Letting $C$ be the right hand side, this proves \eqref{eq:I} and thus (up to the approximations mentioned throughout the proof) concludes the proof of Lemma \ref{L:1pt}.
\end{proof}
As mentioned, similar but more technically slightly more involved arguments for the second moment can be used to prove \eqref{eq:goalbridge}. From this Theorem \ref{T:Weyl} also follows, as already explained. 
\end{proof}

\subsection{On-diagonal heat kernel behaviour}

As mentioned above, the proof of Theorem \ref{T:Weyl} relies on the following asymptotics shown in the proof: for any open set $A\subset \Sigma$, 
\begin{equation}\label{eq:summaryHT}
\mathbf{H}_A(t) = \int_A \mathbf{p}_t(x,x) \mu_\ph(\dd x) \sim \frac{c_\gamma}{t} \mu_\ph(A),
\end{equation}
in probability (in the sense that the ratio of the two sides tend to one in probability as $t\to 0$). This raises the following question: 
\begin{equation}\label{eq:pointwise}
\text{ Does it hold that } \mathbf{p}_t(x,x) \sim \frac{c_\gamma}{t} ,
\end{equation}
pointwise (for $\mu_\ph$-a.e. $x \in \Sigma$) in probability as $t \to 0$? If so, \eqref{eq:pointwise} would be the simplest explanation for \eqref{eq:summaryHT}. Let us mention two arguments (one against, one in favour) concerning the question in \eqref{eq:pointwise}. 

\begin{itemize}

\item When $t \to 0$, whether or not \eqref{eq:pointwise} holds depends on the geometry of the field very close to $x$. As this geometry fluctuates wildly from point to point, the pointwise behaviour of $\mathbf{p}_t(x,x)$ may be expected to also fluctuate wildly from point to point, so expecting \eqref{eq:pointwise} to hold is too naive. 

\item On the other hand, when we restrict to typical points for LQG (this is why we required \eqref{eq:pointwise} to hold only for $\mu_\ph$-a.e. $x \in \Sigma$) the geometry of the field $\ph$ in the vicinity of $x$ has a well-defined law, so after all \eqref{eq:pointwise} is not so unlikely as it might seem at first sight.  
\end{itemize}

In \cite{BW}, we were able to study this question by considering \textbf{annealed asymptotics} where we average over the law of the GFF. 
 Using the tools developed in that paper and discussed above in the proof of Theorem \ref{T:Weyl} we were able to show the following result on the Laplace transform of $\mathbf{p}_t (x,x)$: 
\begin{prop}[Theorem 1.4 in \cite{BW}] \label{P:hk}
Let $x$ be sampled from $\mu_\ph$. Then as $\lambda \to \infty$, 
$$
  \lambda \int_0^\infty e^{-\lambda t} t \mathbf{p}_t (x,x) dt  \Rightarrow J
$$
where $\Rightarrow$ denotes convergence in distribution (with respect to the law of $\ph$), where $J$ is a non-constant random variable, supported on $(0, \infty)$. 
\end{prop}

The convergence is in distribution, meaning when we average over the law of $\ph$ -- i.e., this is an annealed convergence in distribution, as mentioned above. This result shows that the Laplace transform of $\mathbf{p}_t (x,x)$ is approximately $(1/ \lambda)$ times a nontrivial random variable, and so strongly \emph{suggests} the following conjecture: 

\begin{conj}\label{C:hk}
As $t\to 0$, we have the following convergence in distribution, $t\mathbf{p}_t (x,x)\Rightarrow X $, where $X$ is a nontrivial random variable. 
\end{conj}

Once again these are annealed asymptotics, where we average over the law of the GFF. The reason why Proposition \ref{P:hk} does not imply \eqref{C:hk} is that we do not know how to show that the convergence of the Laplace transform in distribution implies convergence of the family of random variables in distribution. (That is, we would need a version of Theorem \ref{L:tauberian} for convergence in distribution rather than in probability). However, the same results (Proposition \ref{P:hk} and Lemma \ref{L:tauberian}) are enough to answer the question in \eqref{eq:pointwise}: namely, it is \emph{not possible} that  $t\mathbf{p}_t (x,x) \to c$ in probability as $t\to 0$. Indeed, if we assumed such a convergence in probability for contradiction, we could apply Lemma \ref{L:tauberian} and this would contradict the fact that the random variable $J$ in Proposition \ref{P:hk} is nonconstant.

\medskip As a consequence of work currently in preparation with J. Klein \cite{BK} we are in fact able to resolve Conjecture \ref{C:hk}, and the random variable $X$ has furthermore a concrete description. 
\begin{thm}\label{T:hk}
As $t\to 0$, we have the following convergence in distribution, 
$$t\mathbf{p}_t (x,x)\Rightarrow X, $$ 
where $X = \mathbf{p}_1^{\cC} (0,0)$ and $\mathbf{p}_t^{\cC} (x,y)$ denotes the heat kernel on the $\gamma$-quantum cone. 
\end{thm}

 We refer to \cite[Definition 7.16]{BP} for the definition of quantum cones. While it would take us too far to give a proof of the above theorem (or even a sketch of a proof) of this result, let us mention one key idea in its proof. Essentially, the proof of this theorem exploits the exact invariance in law of quantum cones under rescalings which blow up their mass in any given domain by a fixed factor $e^{k\gamma}$, where $k >0$. It turns out that such a rescaling also speeds up the corresponding Liouville Brownian motion by the same factor $e^{\gamma k}$: this expresses the fact that, in two dimensions, the ``time'' for Brownian motion is essentially identical to  the ``volume'' of the region in which this particle moves (for instance, the time that it takes for a Liouville Brownian motion to leave a given region is roughly the mass of that region). This result in the following exact identity in law:
 $$
\mathbf{p}_t^{\cC}(x,x) = \frac1t\mathbf{p}_1^{\cC}(x,x).
$$

 \begin{remark} Let us make a few remarks on what this means concerning connections between this result and the Weyl law. 

\begin{enumerate}
\item The negative answer to \eqref{eq:pointwise} shows that the reason why \eqref{eq:summaryHT} holds is quite subtle. While it is not true that the asymptotics \eqref{eq:pointwise} holds pointwise, the heat trace behaves as if it did; this can only be caused by the facts that the fluctuations occurring from point to point (described implicitly in Theorem \ref{T:hk}) cancel each other out when integrating over an open set. 

\item The above suggests (and it is in fact possible to prove) the following interpretation of the Weyl constant $c_\gamma = 1/ [\pi (2- \gamma^2/2)]$: namely, 
\begin{equation}\label{eq:cgamma_cone}
\E [ \mathbf{p}^{\cC}_1(0,0)] = c_\gamma.
\end{equation}
Problem \ref{Prob:hk} below asks for further explicit calculations associated with the random variable $ \mathbf{p}^{\cC}_1(0,0)$ -- e.g., what are its moments of arbitrary order? 
\end{enumerate} 
 
 \end{remark}

The work \cite{BK} also describes the scaling behaviour of the next term of the heat trace, which features a surprising connection with the so-called KPZ formula. This will be discussed below. 

\section{Open problems}

We give below a selection of problems which we believe to be of fundamental importance in the geometric study of LQG. Some have already appeared in \cite{BW}; others appear here for the first time to our knowledge.  

\subsection{No symmetries}

Our first problem is perhaps the biggest in the entire list below, and addresses the question of Mark Kac, \emph{can you hear the shape of a drum?}, in the LQG context. Perhaps surprisingly, our conjecture is that the answer is (a.s.) yes. 

\begin{problem}\label{P:kac} One \emph{can} hear the shape of LQG. In other words,  given a domain $\Sigma$ and the eigenvalue sequence $(\bl_n)_{n\ge 1}$ of LQG, one can measurably recover the underlying Gaussian free field $\ph$. More precisely, there exists a measurable function $\Phi = \Phi_\Sigma: \R^\N \to H^{-1} (\Sigma)$ such that 
$$
\Phi_\Sigma( (\bl_n)_{n\ge 1}) = \ph, \text{ a.s.} 
$$
In fact, we conjecture that it is possible to measurably recover the pair $(\Sigma, \ph)$ up to isometry, i.e., up to equivalence of quantum surfaces, in the sense of \cite{DS}. 
\end{problem}

This conjecture is likely very hard. Note that a remarkable result of Zelditch \cite{Zelditch} shows that spectral determination is \emph{generically} possible, which essentially suggests that the set of potential counterexamples to the above conjecture should have ``measure zero''.

\medskip A possibly simpler problem with a similar flavour asks about the dimension of the eigenspaces. In the presence of symmetries (e.g. on the sphere, or tori with appropriate sidelengths) it is possible that the same eigenvalue corresponds to several orthogonal eigenfunctions. However, in the context of LQG we expect that no such thing is possible:

\begin{problem}
For each $n\ge 1$, let $\dim (E_{\bl_n})$ denote the dimension (in the sense of vector spaces) of the eigenspace associated to the eigenvalue $\bl_n$ in the bounded domain $\Sigma$. Show that 
$$
\dim ( E_{\bl_n}) = 1, a.s.
$$
for all $n\ge 1$. 
\end{problem}

\subsection{Extensions of the Weyl law}

Let us start with what should be an easy problem. In Theorem \ref{T:Weyl} we described the Weyl law when the domain $\Sigma$ is bounded. However, it is perfectly possible to have domains $\Sigma$ which are unbounded and yet the corresponding semigroup (or Green function) of Liouville Brownian motion is trace-class (cf. \cite[Chapter 1.9]{Davies}, for instance Example 1.9.5, in the deterministic case).

\begin{problem}
What can be said about the asymptotics of the eigenvalue counting function and the heat trace when the domain is unbounded and the Green function is trace-class?
\end{problem}

A more substantial problem concerns the critical case $\gamma =2$. In that case it is well known that one can define (through a more careful renormalisation than Theorem \ref{T:GMC}) a measure corresponding $\mu^*_\ph (\dd x) = e^{2 \ph (x) - 2\E[ \ph(x)^2] } \dd x$ \cite{DRSV1, DRSV2} and a corresponding Liouville Brownian motion (see \cite{RV_critical}, though it should be noted that the construction in that paper contains a flaw which was recently fixed by Lacoin \cite{Lacoin}). 

\begin{problem}
Does critical Liouville Brownian motion have a continuous or discrete spectrum? Does a Weyl law hold in this case? 
\end{problem}

We note that, even though the value of the constant $c_\gamma$ in Theorem \ref{T:Weyl} blows up when $\gamma \to 2$, the right hand side of Theorem \ref{T:Weyl} is in fact well behaved: by a result of Aru, Powell and Sepulveda \cite{APScritical}, $$\lim_{\gamma \to 2} c_\gamma \mu_\ph (\Sigma) = c_* \mu^*_\ph(\Sigma),$$ 
where the constant $c_*$ is explicit (and depends on some choices in the normalisation of $\mu_\ph^*$).  This suggests that a Weyl law might actually hold in the critical case.  

Our next two problems are about sharpening the mode of convergence in Theorem \ref{T:Weyl}, which only shows convergence in probability. 
\begin{problem}
Show that 
$$
\frac{\mathbf{N}(\lambda)}{\lambda} \to c_\gamma \mu_\ph(\Sigma)
$$
almost surely as $\lambda \to \infty$. 
\end{problem}

\begin{problem}
Does $\mathbf{N}(\lambda)$ satisfy a Central Limit Theorem (as we average over the realisation of the GFF):
$$
\frac{\mathbf{N}(\lambda) - c_\gamma \lambda \mu_\ph(\Sigma)}{\sqrt{\lambda}} \Rightarrow \cN(0, \sigma^2) 
$$
where $\Rightarrow$ denotes convergence in distribution (with respect to the law of the GFF)? Here $\sigma^2$ should be identified. 
\end{problem}

The above problem asks only about the size of the typical fluctuations of the law of $\mathbf{N}(\lambda)$ at a large but \emph{fixed} $\lambda$ as we average over the law of the GFF. It should not be confused with results of the style of \eqref{eq:Weyl_next} which concerns the next term in the expansion of $\mathbf{N}(\lambda)$ as $\lambda \to \infty$, which is discussed in more details a bit below.

\medskip In the deterministic case, a classical conjecture of P\'olya \cite{Pol1954} states that, for Dirichlet eigenvalues, the eigenvalue counting function lies always below its Weyl limit, i.e., $N(\lambda) \le c_0 \lambda$. P\'olya proved it for so-called tiling domains \cite{Polya_tilingdomains} -- i.e., for domains which can be used to tile the plane, such as a square or a triangle or a hexagon). The case of Euclidean balls was established by Filonov et al. \cite{FLPS2023} only very recently. A closely related result is the Berezin--Li--Yau inequality \cite{Berezin, LiYau} which, informally, says that the conjecture holds for Euclidean domains in a Cesaro sense. The inequalities are reversed for Neumann eigenvalues. 

\medskip It is natural to ask if the same should be expected to hold in the LQG setting. 

\begin{problem}
Do the Dirichlet and Neumann eigenvalues satisfy P\'olya's conjecture or at least a version of the Berezin--Li--Yau inequalities?
\end{problem}

In some sense this problem is difficult because this concerns the behaviour of large eigenvalues as well as those not too large. Numerical experimentation suggests that for relatively low eigenvalues  the behaviour is closer to the Riemannian behaviour than that of LQG, perhaps because the multifractal structure of the geometry is only apparent at short wavelength. Since $c_\gamma> c_0$, this may mean that the conjecture is more easily satisfied for the Dirichlet case than for the Neumann case.

\medskip A different type of extension concerns models \emph{thought} to be related to LQG in the scaling limit. A particularly natural example is that \textbf{random planar maps} (see \cite[Chapter 4]{BP} for a survey of what is generally expected concerning the relations between planar maps and LQG).

\begin{problem}\label{P:Weyl_maps}
Do the eigenvalues of a random planar map $M_n$ obey a Weyl law? That is, do the eigenvalues of the matrix $I- P_n$ (where $P_n$ is the transition matrix of random walk on $M_n$) grow linearly? What is the ``Weyl constant'' (analogue of $c_\gamma$ in Theorem \ref{T:Weyl})?  
\end{problem} 

In the above problem the simplest is to consider the $k$th eigenvalue, letting first $n\to \infty$, then $k \to \infty$. But it also makes sense (and is in fact more interesting) to ask whether one can take $k$ growing with $n$ (the maximal possible value of $k$ is then $k = n$ if there are $n$ vertices). 

One advantage of the above problem is that such questions can be formulated intrinsically, i.e., without appealing to an embedding of the map: the eigenvalues are defined purely in combinatorial terms. Note that the spectral gap (i.e., smallest nonzero eigenvalue of $I-P_n$) is known by work of Gwynne and Hutchcroft \cite{GwynneHutchcroft} and Gwyne and Miller \cite{GwynneMiller}
to be of order $n^{-1/4 + o(1)}$, so the Weyl constant presumably scales like $n^{-1/4}$ with $n$.

\subsection{Spectral expansion of heat trace and boundary length}

As we already mentioned in Section \ref{SS:Weyl_classical}, it is natural to ask about further terms in the expansion of the eigenvalue counting function $\mathbf{N}(\lambda)$, for instance to obtain more spectral invariants. We have already mentioned that even in the classical (deterministic) case, Weyl's conjecture for the next term \eqref{eq:Weyl_next} remains open. However, Courant's theorem \eqref{eq:Courant}) gives an upper bound on the discrepancy and it is natural to ask if the same result holds in LQG:
\begin{problem}
Show that Courant's bound holds in LQG:
\begin{equation}\label{eq:CourantLQG}
\big|\mathbf{N}(\lambda) - c_\gamma \lambda\mu_\ph(\Sigma)\big| = o ( \lambda^{1/2} \log \lambda),
\end{equation}
say, in probability (the left hand side, divided by $\lambda^{1/2} \log \lambda$, tends to zero in probability). 
\end{problem}

In reality, numerical experiments suggest that the true order of deviations of the left hand side of \eqref{eq:CourantLQG} may be much smaller than $\lambda^{1/2}$ and seems to be characterised by a different exponent. In the  work in preparation with J. Klein \cite{BK} we make progress on this question by considering the analogous question in terms of the heat trace. Recall that, in the determinstic case, an expansion of the heat trace \eqref{eq:heattrace_classical2} is well known and formally matches the Weyl conjecture (in particular, the first correction term is associated to the boundary length and scales like $1/\sqrt{t}$) In the LQG case 
we are able to compute explicitly the exponent, which is, surprisingly, not equal 1/2 but is instead connected to the so-called \textbf{KPZ scaling exponent} of $\partial \Sigma$, at least provided that the boundary conditions of the field $\ph$ are not trivial (zero) on $\partial D$. Thus suppose for instance that $\ph$ is a $\gamma$-quantum cone in the entire plan, which we restrict to the bounded domain $\Sigma$, and let $\Delta$ be the KPZ scaling exponent of $\partial\Sigma$, i.e., $\Delta$ solves the equation
\begin{equation}\label{eq:KPZ}
x = \frac{\gamma^2}{4} \Delta^2 + (1- \frac{\gamma^2}{4}) \Delta
\end{equation}
where $x$ is the Euclidean scaling exponent of $\partial \Sigma$, i.e., $\dim (\partial \Sigma) = 2(1-x)$ (thus $x = 1/2$ if $\partial \Sigma$ is smooth). It can be shown in various rigorous senses (\cite{DS, HKPZ, RhodesVargas}, see also \cite[Chapter 3.13]{BP} and references therein for a more thorough discussion) that $1-\Delta$ measures the fractional dimension of $\partial \Sigma$ with respect to that of the ambient space. One of the main results of \cite{BK} is the following result:
\begin{thm}[\cite{BK}]\label{T:traceBK}
Let $\ph$ and $\Sigma$ be as above. Then 
$$
\E( \mathbf{H}(t)) = c_\gamma \E(\mu_\ph(\Sigma)) - t^{-(1- \Delta) + o(1)} 
$$
as $t\to 0$. 
\end{thm}

In the above theorem, the domain $\Sigma$ is not necessarily assumed to be smooth. If it is, then $ x= 1/2$ and solving the quadratic equation \eqref{eq:KPZ} we find 
$$\Delta= \frac12 + \frac{2}{\gamma^2}\left(\sqrt{1+ \frac{\gamma^4}{16}} -1\right) $$ 
Interestingly, if $\Sigma$ is instead a so-called \textbf{quantum disc} (which in the case $\gamma = \sqrt{8/3}$, is thought to be related to the scaling limit of planar maps with the topology of a disc, and which one can think of as bounded by an SLE$_\kappa$ type curve, where $\kappa = \gamma^2 \in (0, 4)$), then the Euclidean scaling exponent $x$ can be computed via the known dimension of these curves by Beffara's result \cite{Beffara}; we find $x = 1/2 - \kappa/16$ and solving the quadratic equation \eqref{eq:KPZ} leads to $\Delta = 1/2$. Thus the next term of the heat trace \emph{is actually} of order $t^{1/2}$ for such quantum discs !

In any case, we conjecture in \cite{BK} that the same scaling governs the next term for the eigenvalue counting function. Our next problem is to prove it. 

\begin{problem}
Show that 
$$
\E [\mathbf{N}(\lambda)] = c_\gamma \lambda\E[\mu_\ph(\Sigma)]-  \lambda^{1- \Delta + o(1)}
$$
\end{problem}
We are not certain if we should expect the above conjecture to hold if we don't take expectations on both sides.

\medskip Going further in the expansion of the heat trace $\mathbf{H}(t)$ than Theorem \ref{T:traceBK} from \cite{BK} seems challenging, but it would be of great interest to pursue this. In particular, we have already mentioned that, in the deterministic situation we know from \eqref{eq:heattrace_classical2} that the second term behaves like the boundary length with scaling $1/ \sqrt{t}$. The fact that the scaling in Theorem \ref{T:traceBK} differs from the classical scaling $1/\sqrt{t}$ has an interesting consequence on the notion of boundary length that \emph{should} be relevant here. Namely, we believe that the only possible notion of boundary length that can be ``heard'' through the heat trace is that of the fractional boundary length 
\begin{equation}
\label{eq:fraclength}
\cL : = \nu(\partial \Sigma), \text{ where } \nu(\dd x) := e^{\gamma ( 1- \Delta) \ph(x)} \sigma(\dd x),
\end{equation}  
and where $\sigma$ is the deterministic length measure on $\partial \Sigma$. That is, $\nu$ is the Gaussian multiplicative chaos measure with \textbf{fractional coupling constant} $\gamma(1-\Delta)$  (one might more naturally have expected the exponent $\gamma /2$ to be relevant, as these are how quantum boundary lengths are defined in LQG, see e.g. \cite{DS}).

\begin{problem}
Is it true that $\mathbf{H}(t) =  c_\gamma \mu_\ph(\Sigma) - t^{-(1- \Delta)}\cL + o(t^{- (1-\Delta)})$ as $t\to 0$?
\end{problem}
\medskip

Going further, we have already mentioned that heat trace asymptotics in the classical/deter\-mi\-nistic setting involve terms related to the $\zeta-$\textbf{regularised determinant of the Laplacian}. 

\begin{problem} Is it possible to define a $\zeta$-regularised determinant of Liouville Brownian motion, and how does it relate (if at all) to asymptotics of the heat trace $\mathbf{H}(t)$ as $t\to 0$?
\end{problem}

One reason this problem is interesting, beyond its potential connection to spectral geometry, is that one can rewrite the Polyakov action functional \eqref{E:Polyakov} defining LQG in terms of the $\zeta$-regularised determinant (see e.g. \cite{TakhtajanZograf, BOW}). In other words one can define LQG by considering powers of the determinants of the Laplacian; see \cite{AngParkPfefferSheffield} and references therein for some partial successes in this direction. (The power to which the determinant is raised is traditionally written in the form $-c/2$, where $c$ corresponds to  the \textbf{central charge} of LQG. It is related to the coupling constant $\gamma$ parametrising LQG via the relation $ c= 25- 6 Q^2$, where (as before) $Q = \tfrac{2}{\gamma} + \tfrac{\gamma}{2}$.) As a concrete consequence, it should be the case that, reweighing the law of LQG with coupling constant $\gamma$ by a power of its zeta-regularised determinant (assuming to the above problem is positive so this indeed well defined), we should obtain an LQG law corresponding to a different coupling constant.





\subsection{Heat kernel asymptotics}

As mentioned throughout this note, asymptotics of the heat kernel are intimately related to spectral geometry. What more can be said about the behaviour of the heat kernel $\mathbf{p}_t(x,x)$ as $t\to 0$? 
We have already ruled out the possibility that
$ \mathbf{p}_t(x,x) \sim \tfrac{c_\gamma}{t}$ (see \eqref{eq:pointwise} and Theorem \ref{T:hk}). It is thus legitimate to ask about \textbf{quenched} asymptotics, i.e., as we freeze the realisation of the GFF $\ph$ and ask about the behaviour of $\mathbf{p}_t(x,x)$ as $t\to 0$. By analogy with what is known on some other random fractals such as random trees (see notably \cite{Croydon} and \cite{CH2008, CH2010}) we expect nontrivial upper and lower logarithmic (or doubly logarithmic) fluctuations. 

\begin{problem}
Does there exists constants $c_\pm\in \R$, $\beta_\pm\in \R$ such that if $x=0$ in a $\gamma$-quantum cone, 
$$
\limsup_{t\to 0} \frac{\mathbf{p}_t(x,x)}{t\ell(t)^{\beta_+} }= c_+, \quad \quad 
\liminf_{t\to 0} \frac{\mathbf{p}_t(x,x)}{t\ell(t)^{\beta_-} }= c_-
$$
where $\ell(t) = \log 1/t$ or $\ell(t) = \log \log 1/t$.  
\end{problem}

By standard Borel--Cantelli arguments, it is likely that the exponents $\beta_-, \beta_+$ should be related to the tail at 0 and at infinity respectively of the random variable $X = \mathbf{p}^{\cC}_1(0,0)$ appearing in Theorem \ref{T:hk}. This raises the following question. 

\begin{problem}\label{Prob:hk}
Can the tail behaviour of $X = \mathbf{p}^{\cC}_1(0,0)$ at zero and infinity be computed? More generally, can moments of $X$ be computed exactly?
\end{problem}

Recall that, as a byproduct of \cite{BK} and Theorem \ref{T:Weyl} we have 
$$
\E( X) = c_\gamma = \frac{1}{\pi(2 - \gamma^2/2)}.
$$
This suggests further exact computations are indeed possible. Note also that it is proved in \cite{BK} that $\E(X^n) <\infty$ for all $n\ge 1$.

\subsection{Quantum chaos}

The Polyakov action defining Liouville quantum gravity has as a ground state (i.e., state of ``maximal probability'') a function $u$ which is a solution to Liouville's equation, known since the pioneering works of Picard \cite{Picard} and Poincaré \cite{Poinc} to describe the conformal factor of a surface of constant negative curvature. For this reason, LQG can be thought of as a theory of random surfaces which are \emph{hyperbolic} to first order. One should therefore expect LQG to display phenomena that are in some sense characteristic of hyperbolic geometry. See \cite{LacoinRhodesVargas, LacoinRhodesVargas2} for a large deviation principle (as $\gamma \to 0$) which connects LQG to solutions of Liouville's equation, and for more on this point of view.

This line of thinking led us in \cite{BW} to propose a bold connection to a phenomenon known as \textbf{quantum chaos}, widely believed to take place for (deterministic) hyperbolic surfaces. Roughly speaking, this phenomenon is the manifestation, at the quantum level, of a certain chaotic behaviour on such surfaces, more precisely the chaoticity of the geodesic flow. To explain this more precisely, the geodesic flow describes the evolution on the surface of a particle with a given initial position and velocity, and moving at unit speed on a geodesic in that initial direction. It is known that on a compact hyperbolic surface $(M,g)$, the geodesic flow is \textbf{chaotic}: the position of this particle quickly forgets its initial position and becomes uniformly distributed (mixed). Quantum chaos predicts, as we have said, that this chaoticity also manifests itself at the quantum level: we now study the position of a ``quantum particle'' instead of one evolving according to the geodesic flow (thought of as a ``classical'' dynamics). Roughly speaking, according to the rules of quantum mechanics, a particle with energy $\lambda_n$ corresponds to a wave-function $f_n$ (the eigenfunction corresponding to $\lambda_n$) and so has a position on $M$ which is a probability distribution given by $|f_n|^2 \dd v_g$, where $\dd v_g$ is the volume form associated to the metric $g$ on $M$, normalised to be a probability distribution on $M$. Thus, according to quantum chaos, as the energy level $\lambda_n \to \infty$ (corresponding to a so-called semi-classical limit), the corresponding position of the quantum particle should also be approximately uniformly distributed. In other words, one should expect
\begin{equation}\label{eq:QC}
|f_n|^2 \dd v_g \Rightarrow \dd v_g
\end{equation}
A weak form of this result (essentially, proving \eqref{eq:QC} along a subsequence of asymptotic density one) was proved in the celebrated works Shnirelman \cite{Shnirelman} and Zelditch \cite{Zelditch}, and later generalised Colin de Verdière \cite{CdV}. 
Its strong form, asserting that there is no need to restrict oneself to a subsequence of density one and so that there are no exceptionally behaving eigenfunctions, is a famous conjecture of Rudnick and Sarnak \cite{RudnickSarnak} known as \textbf{quantum unique ergodicity} (QUE). Lindenstrauss obtained a Fields medal in 2010 partly for his proof of QUE for surfaces related to number theory \cite{Lindenstrauss}.) 

Motivated by the above ideas and the connection between LQG and hyperbolic surfaces, we conjectured in \cite{BW} that eigenfunctions of LQG also display quantum unique ergodicity:

\begin{problem}[\cite{BW}]\label{C:QUE}
 Fix $\gamma \in (0,2)$, and suppose that the LQG eigenfunctions $\ef_n$ are normalised to have unit $L^2 ( \mu_\ph)$-norm. Show that, as $n \to \infty$,
$$
|\ef_n(x)|^2 \mu_\ph(\dd x) \Rightarrow \frac{\mu_\ph(\dd x)}{\mu_{\ph}(\Sigma)} 
$$
 in the weak-$*$ topology, in probability.
\end{problem}

In particular, the conjecture implies that LQG eigenfunctions are a.s. \textbf{delocalised}: i.e., most of their $L^2$ mass is not concentrated in a few microscopic regions, but is instead uniformly spread out. This is in stark contrast with a famous prediction concerning the 2D Anderson model, where microscopic disorder is widely expected to lead to eigenfunction localisation (no matter how small the strength of the disorder). In fact, this conjecture (if true) would mark the first instance of delocalisation in 2D random geometry.

\medskip Quantum chaos is also expected to manifest itself on the spacing statistics of eigenvalues. Indeed, in localised situations (i.e., in the absence of quantum chaos) one expects eigenfunctions to be roughly independent since they are determined by the randomness in a very localised region of space. Therefore it is natural to expect that the eigenvalues will be roughly independent and, once ordered, should therefore correspond (after scaling so the spacing between consecutive eigenvalues is of order one) to a Poisson point process. By contrast, in delocalised situations (when quantum chaos takes place) it is expected that eigenvalues repel each other and display statistics characteristic of random matrix ensembles, more precisely the Gaussian Orthogonal Ensemble (GOE). This remarkable prediction was made in \cite{BGS1984} and remains a fundamental problem in the study of hyperbolic surfaces. See Sarnak \cite{Sarnak} for a discussion of this problem (and much more) and Rudnick \cite{Rudnick} for recent partial progress on this in the context of Weil-Petersson random hyperbolic surfaces. In the LQG context, we conjectured in \cite{BW} (inspired by the same connection to quantum chaos as above) that local eigenvalue statistics are also governed by the Gaussian Orthogonal Ensemble (GOE) of random matrices (see e.g. \cite{AGZ, Meh2004} for an introduction). For instance, in the concrete example of level spacing distribution of eigenvalues, we conjectured:

\begin{figure}
\begin{center}
\includegraphics[width=0.6\textwidth]{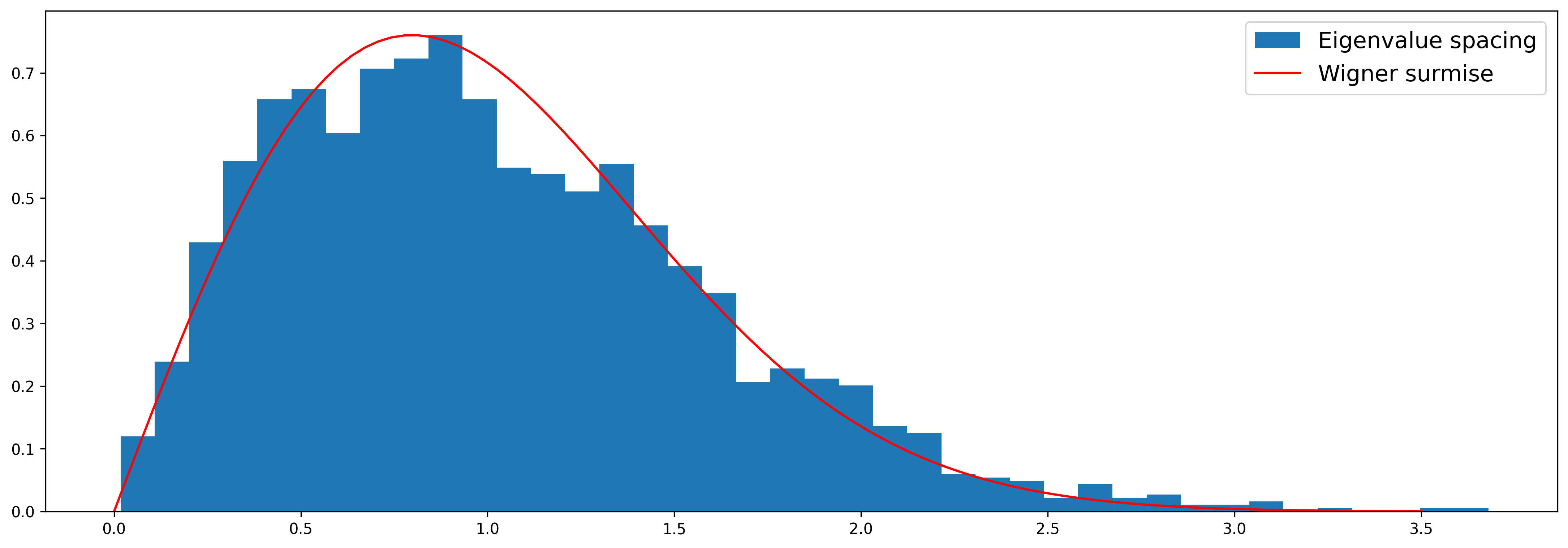}
\end{center}
\caption{Empirical spacing distribution based on the first 2000 LQG eigenvalues (blue) versus GOE statistics (red). Picture from \cite{BW}.}
\label{fig:ev-spacing}
\end{figure}

\begin{problem}\label{C:Sar} 
For each $x \ge 0$, show that 
\begin{equation}\label{CSar}
\frac1{N} \sum_{j=1}^{N} 1_{\{ c_\gamma \mu_\ph(\Sigma)(\boldsymbol{\lambda}_{j+1} - \boldsymbol{\lambda}_j) \le x\} } \xrightarrow[N\to\infty]{p} F_{\emph{GOE}}(x),
\end{equation}
where $F_{\emph{GOE}} (x)$ is the GOE \emph{level-spacing distribution}. 
\end{problem}

Note that the rescaled eigenvalue gap $c_\gamma \mu_\ph(\Sigma)(\boldsymbol{\lambda}_{j+1} - \boldsymbol{\lambda}_j)$ is considered above since it is approximately  equal to $1$ on average in the long run, as established by our Weyl's law (Theorem \ref{T:Weyl}). 
The spacing distribution $F_{\text{GOE}}$, also known as Gaudin distribution (for $\beta = 1$) in the literature, may be expressed in terms of a Fredholm determinant involving the Sine kernel \cite{Gau1961} as well as the Painlev\'e transcendents  \cite{ForresterWitte}. 
See  Figure \ref{fig:ev-spacing} for a comparison between the empirical LQG eigenvalue spacing distribution and our GOE conjecture.\\

\begin{figure}
\begin{center}
\includegraphics[width=0.5\textwidth]{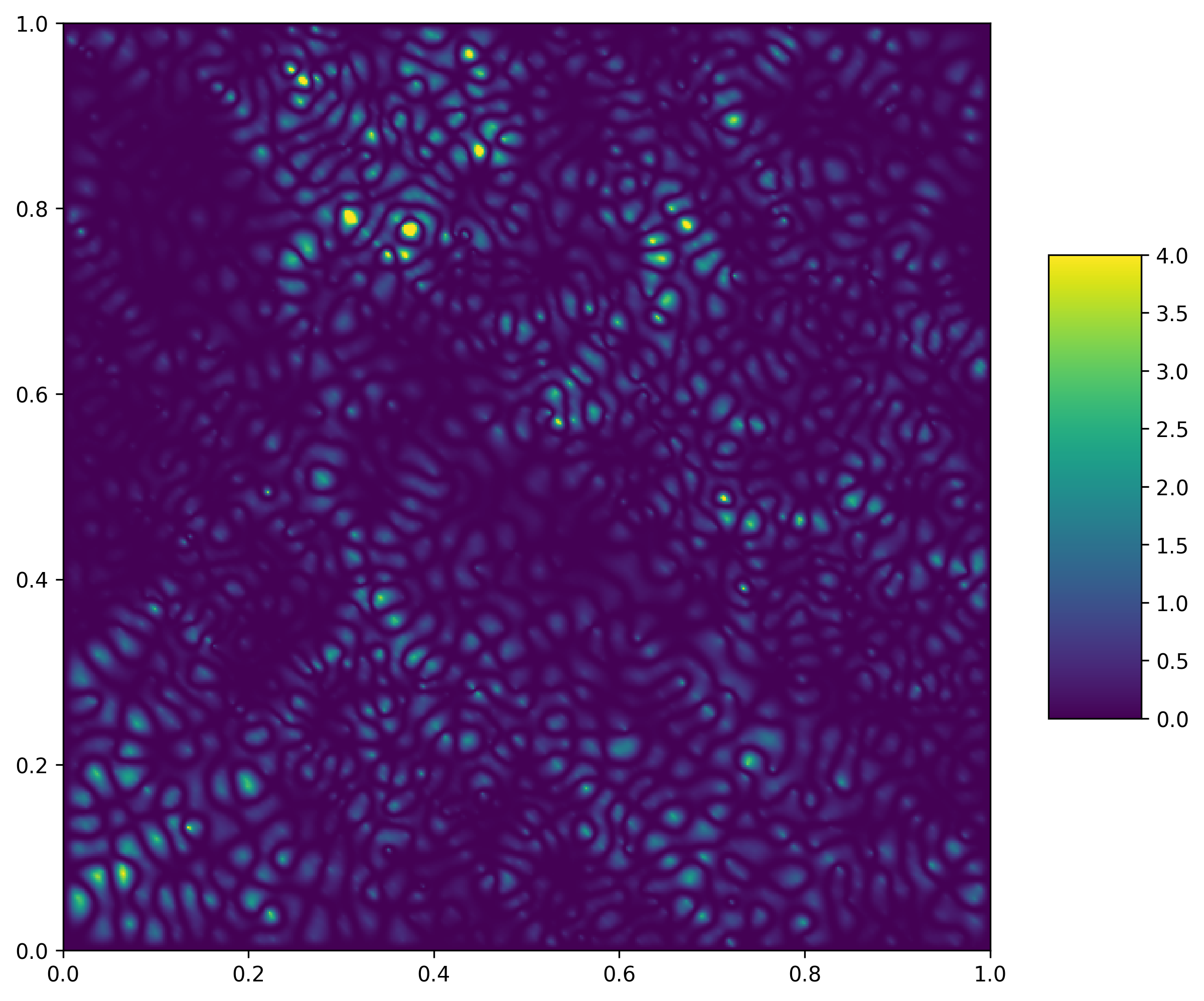}
\end{center}
\caption{Heat map of $|\mathbf{f}_n|^2$ for $n = 2000$. Picture from \cite{BW}.}
\label{fig:Berry}
\end{figure}

A final aspect of quantum chaos concerns the local behaviour of eigenfunctions, see Figure \ref{fig:Berry}. What do they look like in the vicinity of a given typical point, in the limit as $n\to \infty$? In 1977, Berry \cite{Ber1977} proposed that when quantum chaos is relevant (i.e., in delocalised situation and thus in particular for hyperbolic surfaces) the eigenfunctions should be described a certain random field, now known as the Berry random wave model. This is a Gaussian centered random field $(b(x))_{x\in \R^2}$ in the plane, whose covariance function is, by definition, 
\begin{equation}
\E[ b(x) b(y)] = J_0 ( \| x-y \|),
\end{equation}
where $J_0$ is the Bessel function of the first kind and of order 0. This leads us to the following question:

\begin{problem}\label{Prob:Berry}
Let $x$ be sampled from $\mu_\ph$, and let $r_n>0$ be a rescaling (to be determined). Is it true that $\mathbf{f}_n ( x + r_n \cdot)$ converges to Berry's random wave model as $n\to \infty$?
\end{problem}

Here the rescaling $r_n$ is related to wavelength of the eigenfunction $\mathbf{f}_n$, but may also depend on the local behaviour of $\ph$ near $x$. This problem may be simplified by averaging over a given spectral window: instead of considering a single eigenfunction $\mathbf{f}_n$  associated to $\bl_n$, one considers a window of the form $[\bl_n , \bl_n + \omega_n]$ and tries to prove Problem \ref{Prob:Berry} for the random function $$\tilde{\mathbf{f}}_n: = \sum_{k: \bl_k\in [\bl_n, \bl_n + \omega_n]} X_k \mathbf{f}_k,$$
 where $(X_k)_{k\ge 0}$ are i.i.d. standard Gaussians.

\subsection*{Acknowledgements}

It is a pleasure to thank my collaborator Mo Dick Wong for many discussions related to spectral geometry and LQG while we were working on the paper \cite{BW}. Likewise, I thank Jakob Klein (as well as Tomas Alcalde) for our discussions during our work on \cite{BK}. Thanks also to Joffrey Mathien for comments on a draft of this paper. Finally, I gratefully acknowledge the support from the Austrian Science Fund (FWF) grants 10.55776/F1002 and 10.55776/PAT1878824. 







\begin{large}
\bibliographystyle{alpha}  
\bibliography{biblio_ERC2023.bib}
\end{large}

\end{document}